\documentclass[12pt]{article}
\usepackage{amsmath,amssymb,amsthm,latexsym,amscd}
\usepackage{hyperref}
\newtheorem{thm}{Theorem}
\newtheorem{lem}{Lemma}
\newtheorem{prop}{Proposition}
\newtheorem{cor}{Corollary}
\newtheorem{rmk}{Remark}
\newtheorem{defin}{Definition}
\newtheorem{prm}{Problem}

\title{The complexity of the first-order theory of the pure equality\footnote{\indent{\em ACM classes:} F.1.1, F.2.3, F.4.3}}
\author{Ivan V.
Latkin\footnote{lativan@yandex.ru}}
\date{}
\begin{document}
\maketitle

\begin{abstract} We will find a lower bound on the recognition complexity of the 
theories that are nontrivial relative to some equivalence relation (this relation 
may be equality), namely, each of these theories is consistent with the formula, 
whose sense is that there exist two non-equivalent elements. However, at first, we 
will obtain a lower bound on the computational complexity for the first-order theory 
of Boolean algebra that has only two elements. For this purpose, we will code the 
long-continued deterministic Turing machine computations by the relatively short-
length quantified Boolean formulae; the modified Stockmeyer and Meyer method will 
appreciably be used for this simulation. Then, we will transform the modeling
formulae of the theory of this Boolean algebra to the simulation ones of the
first-order theory of the only equivalence relation in polynomial time.

Since the computational complexity of these theories is not polynomial, we 
obtain that the class $\mathbf{P}$ is a proper subclass of $\mathbf{PSPACE}$ 
(Polynomial Time is a proper subset of Polynomial Space).

\end{abstract}

\section{Introduction}

At the beginning, we recall some designations. A function $\exp_k(n)$ is 
called {\em $k$-iterated} or {\em $k$-story} exponential, if, for every natural 
$k$, it is calculated in the following way: \ $\exp_0(n)\!=\!n$, \ 
$\exp_{k+1}(n)\!=\!2^{\exp_{k}(n)}$. The length of a word $X$ is denoted by 
$|X|$, i.e., $|X|$ is the number of symbols in $X$. If $A$ is a set, then   
$|A|$ denotes its cardinality; "$A \rightleftharpoons\mathcal{A}$" means 
"$A$ is a designation for $\mathcal{A}$"; and $\exp(n)\rightleftharpoons
\exp_1(n)$.

\subsection{\large Problem statement}\label{s1.1}
The results on the complexity of recognition (or computational complexity) for 
many of the decidable theories are well-known \cite{Fish-Rab},\cite{Mey73}--
\cite{St74}. We recall only some of these results concerning first-order theories.

Any decision procedure has more than an exponential complexity for the theory 
$ThRLC$ of the field $\mathbb{R}$ of real numbers, and even for $Th(\mathbb{R},+)$ 
\cite{Fish-Rab}. More precisely, there exists a rational constant $d_1>0$, such 
that if $P$ is a deterministic Turing machine which recognizes the theory $ThRLC$ 
(or $Th((\mathbb{R},+))$), then $P$ runs for at least $2^{d_1|\varphi|}$ steps 
when started on input $\varphi$, for infinitely many sentences $\varphi$. 
So the {\it complexity of recognition} for these theories (which corresponds to 
the concept of {\it inner complexity} as defined in \cite{Rab}) is more than 
$\exp(d_1n)$, here and below, the variable $n$ is the length of the input string; 
and the letter $d$ with subscripts denotes a suitable constant. In other words, 
$ThRLC$ and $Th(R,+)$ do not belong to $DTIME(\exp(d_1n))$. For Presburger 
arithmetic $PAR$ (the theory of natural numbers with addition) and for Skolem 
arithmetic $SAR$ (the theory of natural numbers with multiplication), the 
recognition complexity is more than a double exponential: $PAR,SAR\!\notin\!
DTIME(\exp_2(d_2n))$. For the theory of linearly ordered sets $ThOR$, the 
computational complexity is very great \cite{Mey74}: \ $ThOR\!\notin\!
DTIME(\exp_{\lfloor d_3n\rfloor}(n))$, where $\lfloor y\rfloor$ is the integer 
part of a number $y$.  

It is quite natural to expect that if we go beyond the confines of logical 
theories of the first order, then we can see more impressive lower bounds on 
the recognition complexity. An example of such an estimate is the lower bound for 
the weak monadic second-order theory of one successor $WSIS$, other examples can 
be found in \cite{Rab,Mey73,St74,ComHen}. However, according to the author, the 
most impressive estimate of this kind was obtained by Vorobyev~S.G. \cite{Vor} 
for the type theory $\Omega$, which is a rudimentary fragment of the theory of 
propositional types due to Henkin: \ $\Omega\!\notin\!DSPACE(\exp_{\infty}
(\exp(d_4n)))$, hence \ $\Omega\!\notin\!DTIME(\exp_{\infty}(\exp(d_4n)))$, 
where the function $\exp_{\infty}$ is recursively defined by $\exp_{\infty}(0)\!=
\!1$ and $\exp_{\infty}(k\!+\!1)\!=\!2^{\exp_{\infty}(k)}$, i.e., this lower bound 
has the exponentially growing stack of twos.

The theories with such fantastic computational complexity had been named {\it 
nonelementary} in \cite{ComHen} and \cite{Vor}. In this sense, we can say that 
the theory $\Omega$ is one of the most nonelementary theories at the present 
moment.

And what is the recognition complexity of the simplest (in the semantic and 
syntactical sense), but non-trivial theories? Should it be polynomial? In other 
words, shall such theories be most elementary regarding recognition?

One of the simplest theories is the first-order theory of the algebraic structure 
of two elements with a unique equality predicate. We will see in Section \ref{s7} 
that even this theory does not have a polynomial upper bound of computational 
complexity. We will in passing obtain the lower bounds on the recognition 
complexity of the theories that are nontrivial relative to some equivalence 
relation $\backsim$, namely, these theories have models with at least two elements 
that are not $\backsim$-equivalent. Obvious examples of such theories are the 
theories of pure equality and of one equivalence relation.

Since the lower bound on the computational complexity of these theories is not 
polynomial, we obtain that the class $\mathbf{P}$ is a proper subclass of 
$\mathbf{PSPACE}$.

\subsection{Used methods and the main idea}\label{s1.2}
The lower bounds on the computational complexity for the theories mentioned in the
previous subsection and some others were yielded by the techniques of the 
efficient reducibility of the machines to the formulae in \cite{Fish-Rab},
\cite{Mey73}--\cite{St74},\cite{Vor}, or more precisely, by methods  
of the immediate codings of the machine actions. The essence of these methods
\footnote{We will call this the Rabin and Fischer method or the technique for 
modeling of computations by means of formulae} is as follows \cite{Rab}. Let $T$ 
be the theory under study, written in the signature (or {\it underlying 
language} \cite{Rab}) $\sigma$. Assume that, for any input string $X$ and 
every program $P$ of the Turing machine, one can write a sentence $S(P,X)$, of 
$\sigma$, satisfying the following conditions. 
There exist a constant $d\!>\!0$ and a function $f$ such that: \ 
(i) $|S(P,X)|\!<\!d(|X|\!+\!|P|)$; \ (ii) $S(P,X)\!\in\!T$ if and only if a 
computation by the program $P$ accepts the input $X$ in fewer than $f(|X|)$ 
steps; \ (iii) the formula $S(P,X)$ can be effectively constructed from $X$ 
and $P$ in fewer than $g(|X|\!+\!|P|)$ steps, where $g(k)$ is a fixed 
polynomial. If $f(k)$ is a function growing at least at exponential rate, 
then under the above conditions, there exist a constant $C\!>\!0$ and infinitely 
many sentences $\varphi$ of $\sigma$, for which every Turing machine requires 
at least $f(C|\varphi|)$ steps to decide whether $\varphi\!\in\!T$, i.e., 
$T\!\notin\!DTIME(f(Cn))$. 

The proof of the last statement is based on a well-known diagonal argument, 
though we will below scrutinize this method in more detail and in a somewhat 
more general form than this was done in the previous paragraph or in  
Subsection 4.1 in \cite{Rab}. We need the more general form of this technique 
for the following reason.  

Our main purpose is to evaluate the computational complexity of an equality  
theory $Th\mathcal{E}$ (Section \ref{s7}). However, at first, we will obtain 
a lower bound on recognition complexity for the first-order theory of Boolean 
algebra $\mathcal{B}$ that has only two elements, using the Rabin and Fischer 
method. Then, we will construct a polynomial reduction of the modeling formulae 
of $Th\mathcal{B}$ to the simulation ones of $Th\mathcal{E}$. In Subsection 
\ref{s8.1}, we will explain why such a succession of actions is applied.

But the first-order theory of two-element Boolean algebra has a very weak 
expressive ability. Therefore, the modeling sentence for this theory, i.e., 
the formula possessing property (ii) from the method described above, does not 
turn out to be very short, it may have not a linear restriction on its length 
(see Subsection \ref{s6.4} for more details). Furthermore, $Th\mathcal{B}$ 
\ is so poor and meager that there can, in general, be a doubt about the very 
possibility of the simulation of the sufficiently long computations by means 
of the relatively short formulae of this theory. 

Nevertheless, such modeling was well-known a long time ago. Back in 1971, 
Cook~S.A. constructed a quantifier-free Boolean formula $A(M,w,Q)$, which 
simulates the actions of a non-deterministic Turing machine $M$ on an input $w$ 
during $Q(|X|)$ steps (in the proof of $\mathbf{NP}$-completeness of the problem 
SAT \cite{Ck}, see also the proof of Theorem 10.3 in \cite{AHU}). This formula may 
seem too long, since it has length $\mathcal{O}(|M|Q^2(|w|))$. But it should be 
taken into account that even if $Q(n)$ is a polynomial, then the number of 
different branches of the computational process of a non-deterministic machine can 
be exponential. And Cook's formula allows us to derive a description of any such 
branch. Thus, Cook's formula, which has polynomial length, contains an implicit 
description of the exponential number of possible calculations\footnote{This remark 
is not as trivial as it might seem. Many students, having familiarized themselves 
with the proof of Cook's theorem, seriously believe that Cook's formulae contain a 
description of \emph{ALL} configurations, which can arise during the operation of 
the machine M, in an explicit form.}.

Next, Stockmeyer~L.J. and Meyer~A.R. showed in 1973 that a language  
$TQBF$\footnote{The problem corresponding to this language is designated as $QBF$, 
or sometimes $QSAT$.} consisting of the true quantified Boolean formulae 
is polynomially complete in the class $\mathbf{PSPACE}$ \cite{S-M,G-J}. This 
implies in particular that for every language $\mathcal{L}$ in this class, there 
is an algorithm, which produces a quantified Boolean formula for any input string 
in polynomial time; and all these sentences model the computations that 
recognize $\mathcal{L}$ and use the polynomial amount of space. Namely, 
each of theirs is true if and only if the given input string belongs to the 
language under study; at that, the long enough computations are simulated, 
seeing that the polynomial constraint on memory allows the machine to run during 
the exponential-long time \cite{AHU,G-J,ArBar}.

Stockmeyer~L.J. and Meyer~A.R. have employed the highly ingenious technique 
for the implementation of this si\-mu\-lation (see the proof of Theorem 4.3 in 
\cite{S-M}). Their approach permits writing down a polynomially bounded formula 
for the modeling of the exponential quantity of the Turing machine steps 
provided that one step is described by the formula, the length of which is 
polynomial. One running step of the machine is described in \cite{S-M} 
by the Cook's method formula $A_{0,n}(M,w)$; one can regard this $A_{0,n}(M,w)$ as 
a subformula of the above-mentioned sentence $A(M,w,Q)$, but $n\!=\!q(|X|)$ is the 
size of used memory now --- see Subsection \ref{s8.2} for more details. There 
exists a Boolean $\exists$-formula, which corresponds to the Cook's method 
formula. We will also name this $\exists$-formula as Cook's formula.

We intend to modernize the elegant construct of Stockmeyer and Meyer and to 
bring it into play for obtaining our purpose. But we will model the running steps 
of a machine by the more complicated formulae. This complication is caused due 
to the fact that Cook's formula $A_{0,n}(M,w)$ is very long for our aim --- it is 
far longer than an amount of the used memory. Really, it has a subformula that 
consists of one propositional variable $P_{i,j,t}$ (see also the proof 
of Theorem~10.3 in \cite{AHU} and also Section \ref{s8}). This variable is true 
if the $i$th cell contains symbol $\sigma_j$ of the tape alphabet at the instant of 
time $t$. However, suppose that each of the first $T+1$ squares of tape contains 
the symbol $\sigma_0$ at time $t$, the remaining part of the tape is empty. This 
simple tape configuration (or \emph{instantaneous description} \cite{Ck,S-M,AHU}) 
is described by the formula that has a fragment \ $P_{0,0,t}\wedge P_{1,0,t}
\wedge\ldots\wedge P_{T,0,t}$, and this subformula is $2T+1$ in length without 
taking the indices into account. It is impossible to abridge this record, even 
if we try to use the universal quantifier since its application to the indices 
is not allowed within the confines of the first-order theory. Thus, in order 
to describe the machine actions using the exponential amount of space, we need 
Cook's formula, whose length is no less than exponential.

We propose to encode the binary notation of the cell number by a value set of 
special variables $x_{t,0},\ldots,x_{t,n}$, where $n\!+\!1\!\geqslant\!\log_2T
$ (see Subsections \ref{s4.1}, \ref{s6.2}, and \ref{s8.2} for further details). 
So we need $\mathcal{O}(n)$ symbols (without indices) for the describing of one 
cell, and $\mathcal{O}(n^2)$ ones for the assignment of the whole input string $X
$, if $n\!=\!|X|$. Then we can describe one running step of the machine, which 
uses $T\!\rightleftharpoons\!\exp(|X|)$ memory cells on input $X$, with the aid of
a formula that is not more than $\mathcal{O}(n^{3})$ in length. The main idea of 
so brief a describing consists of the following: merely one tape square can change 
on each of the running steps, although the whole computation can use the 
exponential amount of memory.\footnote{The denoted locality of the actions of 
deterministic machines has long been used in the modeling of the machine 
computation with the help of formulae, see, for example, Lemma 2.14 in \cite{St74} 
or Lemma 7 in \cite{Vor}.} Therefore, it is enough to describe the changes in the 
only cell, and the contents of the remaining ones can be "copied" by applying the 
universal quantifier (see the construction of the formula $\Delta^{\textit{cop}}
(\widehat{u})$ in Subsection \ref{s4.1}). 

At the beginning, we will introduce all variables in great abundance in order 
to facilitate the proof, namely, the variables will have the first indices $t$ 
from 0 to $T$. Next, we will eliminate many of the variables using the 
modified method of Stockmeyer and Meyer --- see Subsections \ref{s4.3.2}, 
\ref{s6.2}, \ref{s8.2} for further details. A final modeling formula will only 
contain those of variables for which \ $0\!\leqslant\!t\!\leqslant\!n$ \ or \ $t\!
=\!T$ hold.

The description of the initial configuration and the condition of the 
successful termination of computations have a length of $\mathcal{O}(n^{3})$, if 
we anew use the quantifiers; hence the entire formula, which simulates the first 
$\exp(n)$ steps of the computation of the machine $P$, will be 
$\mathcal{O}(|P|\cdot n^{3})$ in length (taking into account the indices). 

Therefore, we need to slightly strengthen the Rabin and Fischer method, so that 
it can also be applied in the case of a non-linear estimate for the length of 
the modeling formula.

\subsection{The paper structure} \label{s1.3}
The generalized Fisher and Rabin method is adduced in Section \ref{s2}. The 
degree of its usefulness and novelty is discussed in Remark \ref{rem1}. 
Section \ref{s3} contains an exact formulation of the main theorem (Theorem 
\ref{main}), its primary corollaries, and some preparation for that and for 
the proof of this theorem. Sections \ref{s4}--\ref{s6} are devoted to the 
proof of the main theorem. The lower bound on the computational complexity of 
the theories, which are nontrivial relatively to some equivalence relation, in 
particular, equational-nontrivial, will be yielded in Section \ref{s7}. In 
Section \ref{s8}, we will discuss the obtained results and consider the used 
methods in greater detail, comparing theirs with other approaches to the 
simulation of computation. The short list of the open problems concludes the 
paper.

\section{The generalized Fischer and Rabin method}\label{s2}

We will describe this method in the most general form. 

\subsection{Auxiliary notions}\label{s2.1}
We will need some new concepts.

\begin{defin} Let $P$ be a program of the Turing machine; $k$ be a number of its 
tapes; and $q_b\alpha_1\alpha_2\ldots\alpha_k\to q_j\beta_1\beta_2\ldots\beta_k$ be  
an instruction of this program. We will call this instruction {\it explicitly 
non-executable} and the internal state $q_b$ {\it inaccessible (for $P$)}, if the 
program $P$ does not contain the instructions of the form $q_l\gamma_1
\gamma_2\ldots\gamma_k\to q_b\delta_1\delta_2\ldots\delta_k$. \end{defin} 

One can easy write such a machine program that it contains some non-executable 
instructions, but all its internal states are accessible. It is evident too that 
one can easily find the explicitly non-exe\-cutable instructions in any program, 
more precisely, all of them can be found in polynomial time on the program length. 
However, the detection of the non-executable instructions, whose internal states 
are accessible, can be a very difficult task in some cases. 

Let us assume that we have removed all the explicitly non-executable 
instructions from a program $P$. The eli\-mi\-nation has resulted in some 
program $P_1$. This $P_1$ may again contain some explicitly non-executable 
instructions, for instance, if the instructions $q_l\gamma_1\gamma_2\ldots
\gamma_k\to q_b\delta_1\delta_2\ldots\delta_k$ and $q_b\alpha_1\alpha_2\ldots
\alpha_k\to q_j\beta_1\beta_2 \ldots\beta_k$ belong to $P$, the first of them 
is explicitly non-executable for $P$, and the state $b$ is not included in 
other instructions, then the second instruction is not such in full, although it 
is non-executable for $P$. However, it already is explicitly non-executable 
for the program $P_1$. We can continue this removing process of the explicitly 
non-executable instructions until we obtain the {\it irreducible} program 
$r(P)$ that does not contain such instructions. 

We name the programs $T$ and $P$ {\it monoclonal} if $r(T)\!=\!r(P)$; at that  
$P$, $T$, and $r(P)$ are called the {\it clones} of each other. As 
usual, a Turing machine and its program are designated by a uniform sign. 
Therefore we will say that two Turing machines are {\it monoclonal} if their 
programs are so. 

\begin{lem}\label{cl} (i) There exists a polynomial $h(n)$ such that one can 
write the code of irreducible clone $r(P)$ within $h(|P|)$ steps for every 
program $P$;

(ii) all the tape actions of monoclonal machines are identical with each 
other on the same inputs. \end{lem}%

\begin{proof} It straightforwardly follows from definitions. \end{proof}

\begin{defin}\label{dLUBP} Let $F(n)$ be a function that is monotone increasing 
on all sufficiently large $n$. The function $F$ is called a {\it limit upper 
bound for the class of all polynomials} (LUBP) if, for any polynomial $p$, there 
is a number $n$ such that the inequality $F(m)\!>\!p(m)$ holds for $m\!\geqslant
\!n$, i.e., each polynomial is asymptotically smaller than $F$. \end{defin}

An obvious example of the limit upper bound for all polynomials is a 
$s$-iterated exponential for every $s\!\geqslant\!1$. It is easy to see that 
if $F(x)$ is a LUBP, then the functions $F(x^m)$ and $F(rx)$ are also LUBPs 
for positive constants $m$ and $r$, moreover, the function $F(x)\!-\!F(dx)$ is 
a LUBP for every constant $d$ such that $0\!<\!d\!<\!1$. It follows from this 
that if $T(n)$ is a LUBP, then it grows at least exponentially in the sense that 
is considered in \cite{ComHen}, namely, $T(dn)/T(n)$ tends to 0 as $n$ tends to 
$\infty$. Inverse assertion seemingly is valid too.

\subsection{The generalization}\label{s2.2}
Let us suppose that we want to find a lower bound on the recognition 
complexity of a language $\mathcal{L}$ over alphabet $\sigma$. We, first of all, 
fix a finite tape alphabet $A$ of Turing machines and the number $k$ of their 
tapes. We also fix a certain {\it polynomial encoding} of the strings over the 
alphabet $\sigma$ and of the programs of Turing machines by finite strings of 
symbols (words) over the alphabet $A$, i.e., it is implied that the encoding 
and unique decoding are realized in a polynomial time from the length of an 
object in a natural language.\footnote{This language consists of all words 
over the alphabet $\sigma$ and all the Turing machines programs with the  $k$
tapes and the tape alphabet $A$. An example of such a natural language will be 
described in Subsection \ref{s3.1}. It is implied here and below that the 
numbers of the internal states and other indices are written in decimal 
notation.} We presume also that the used encoding is {\it composite}, namely, 
the code of each instruction in any program is the constituent of the program 
code. The code of an object $E$ is denoted by $c_OE$, i.e., $c_OE\!\in\!
A^{\ast}$, if $E\!\in\!\sigma^{\ast}$ or $E$ is a program. 

\begin{prop}\label{mainPr} Let $F$ be a limit upper bound for all polynomials 
and $\mathcal{L}$ be a language over some alphabet $\sigma$. Suppose that for 
any given program $P$ of a Turing machine and every string $X$ on the input 
tape of this machine, one can effectively construct a word $S(P,X)$ over the 
alphabet $\sigma$ with the following properties:

(i) a code for $S(P,X)$ can be built within time $g(|X|\!+\!|c_OP|)$, where 
$g$ is a polynomial fixed for all $X$ and $P$;

(ii) the word $S(P,X)$ belongs to $\mathcal{L}$ if and only if the Turing 
machine $P$ accepts input $X$ within $F(|X|)$ steps; 

(iii) there exist constants $D,b,s\!>\!0$ such that either the inequalities 
$$(a)\qquad |X|\!\leqslant\!|c_OS(P,X)|\!\leqslant\!D\cdot|c_OP|^b\cdot|X|^s$$ 
\ or the inequalities $$(b)\qquad |X|\!\leqslant\!|c_OS(P,X)|\!\leqslant\!D
\cdot(|c_OP|\!+\!|X|)$$ \ hold true for all sufficiently long $X$, and these 
constants do not depend on $P$, but they depend on the applied encoding.

Then \qquad (1) for every constant $\delta\!>\!0$ and any program $P$, there 
is a number $t_0$ such that the inequality \ $|c_OS(P,X)|\!\leqslant\!D_1
\cdot|X|^{s_{1}}$ \ holds for all of the strings $X$, which are longer 
than $t_0$, where $D_1\!=\!D$ and $s_1\!=\!s\!+\!\delta$ in  case (a) \quad or 
\ $D_1\!=\!(D\!+\!\delta)$ and $s_1\!=\!1$ in case (b); \\
(2) for each $a\!>\!1$ and every deterministic Turing machine $M$, 
which recognizes the language $\mathcal{L}$, there exist infinitely many words 
$Y$, on which $M$ runs for more than \ $F(D_{2}\cdot|c_OY|^{\rho})$ steps for \
$D_2\!=\!(aD_1)^{-\rho}$ and $\rho\!=\!(s_1)^{-1}$.\end{prop}

\begin{proof} \ (1). It is easy to see that $t_0$ is equal to 
$|c_OP|^{b/\delta}$ in case (a); and it equals to $(D/\delta)\cdot|c_OP|$ in 
case (b).
 
(2) In accordance with condition (i), one can assume that a code for $S(P,X)$ 
is written by some machine $M_1$ for all given strings $X$ and $c_OP$.

Let us suppose that there exist numbers $a,t_1$ and a machine $M_2$ such that
$M_2$ determines whether $Y\!\in\!\mathcal{L}$ within $F(D_2\cdot|c_OY|^{\rho})$ 
steps for any string $Y$ over $\sigma$, provided that $|c_OY|\!>\!t_1$ and 
$a\!>\!1$.
 
To proceed to an ordinary diagonal argument, we stage-by-stage construct the   
Turing machine $M$. At the first stage, we write a machine $M_0$, which for a 
given input $X$, determines whether the string $X$ is the code $c_OP$ of some 
program $P$. If not, then $M_0$, as well as the whole machine $M$, rejects 
$X$; else it writes the code $c_Or(P)$ of the irreducible clone $r(P)$. 

At the second stage, $M_1$ joins the running process and writes a word
$c_OS(r(P),c_OP)$. At the next stage, the procedure $M_2$ determines 
whether the string $S(r(P),c_OP)$ belongs to the language $\mathcal{L}$. If it 
does not, then $M$ accepts the input $X\!=\!c_OP$. When $M_2$ gives an 
affirmative answer, then $M$ rejects $X$.

We estimate the running time of $M$ on input $X\!=\!c_OP$. Since $c_O$ is a 
polynomial encoding and Lemma \ref{cl}(i) is valid, there exists a polynomial 
$h_1$ such that the running time of $M_0$ does not exceed $h_1(|X|)$. The 
machine $M_1$ builds $c_OS(r(P),X)$ within $g(|X|\!+\!|c_Or(P)|)\!\leqslant\!
g(2|c_OP|)$ \ steps, since $|c_Or(P)|\!\leqslant\!|c_OP|$; \ the stage $M_2$ 
lasts no longer than \ $F(D_2\cdot|c_OS(r(P),c_OP)|^{\rho})\!\leqslant\!
F((D_1\cdot|c_OP|^{s_{1}})^{\rho}/(aD_1)^{\rho})\!=\!F(|c_OP|/a^{\rho})$ steps 
for $|c_OS(r(P),c_OP)|\!\geqslant\!|c_OP|\!>\!t_1$ by our assumption. Hence, 
the entire $M$ will execute its work within no more than \ $T(P)\!=\!h_1(|c_OP|)
+g(2|c_OP|)+F(|c_OP|/a^{\rho})\!<\!F(|c_OP|)$ steps for all sufficiently large 
$|c_OP|$. 

Let us look at the situation that obtains if as $X$ we take the code of so 
lengthy a clone $\widehat{M}$ of the machine $M$ that the inequalities 
$|c_O\widehat{M}|\!>\!\max\{t_0,t_1\}$ and $T(\widehat{M})\!<\!F(|c_O
\widehat{M}|)$ hold true. 

If $M$ rejects the input $c_O\widehat{M}$, then $M_2$ answers affirmatively, i.e., 
the string $S(r(\widehat{M}),c_O\widehat{M})$ belongs to 
the language $\mathcal{L}$. According to the condition (ii), this means that 
$r(\widehat{M})$ accepts the input $c_O\widehat{M}$ within $F(|c_O
\widehat{M}|)$ steps. Since the machines $M$, $\widehat{M}$, and 
$r(\widehat{M})$ are monoclonal, $M$ does it too. There is a contradiction.

If $M$ accepts $c_O\widehat{M}$ as its input, then the procedure $M_2$ answers 
negatively. Under the sense of the formula $S(r(\widehat{M}),c_O\widehat{M})$, 
this signifies that the machine $r(\widehat{M})$ either rejects $c_O\widehat{M}$ 
or its running time on this input is more than $F(|c_O\widehat{M}|)$. By 
construction and our assumption, the clone $r(\widehat{M})$ cannot operate so long. 
We have again arrived at a contradiction. \end{proof}

\begin{rmk}\label{rem1} Apparently, the generalization of Rabin and Fischer's 
method has been in essence known in an implicit form for a long time. For 
example, it is said in the penultimate paragraph of the introduction of the 
article \cite{Vor} (before the paragraph "Paper outline") that the quadratic 
increase in the length of the modeling formulae implies a lowering of the 
lower bound with $F(n)$ to $F(\sqrt{n})$ (in our notation), when Compton and  
Henson's method is applied. But the author could not find an explicit 
formulation of the statement similar to Proposition \ref{mainPr} for a 
reference, although its analog for the space complexity is Lemma 3 in 
\cite{Vor}. The proof of the proposition is given only for the sake of 
completeness of the proof of Corollary \ref{mainCor}. In addition, Proposition 
\ref{mainPr} in such form is clearly redundant for the proof of this corollary. 
However, the author hopes to apply it in further researches. \end{rmk}

\begin{cor}\label{comRecL} Under the conditions of the proposition \ 
$\mathcal{L}\!\notin\!DTIME(F(D^{-\zeta}\cdot n^{\zeta}))$, where $\zeta\!=\!
s^{-1}$ ($s\!=\!1$ in case (b)). \end{cor} 

\begin{proof} Really, $s_1\!=\!s\!+\!\delta$ and $aD_1\!=\!a(D\!+\!\delta)$ 
tend to $s$ and $D$ respectively, when $a$ tends to one and $\delta$ tends to 
zero. Hence, $\rho\!=\!(s\!+\!\delta)^{-1}$, $n^{\rho}$, and $D_{1}^{-\rho}$ 
accordingly tend to $s^{-1}$, $n^{s^{-1}}$, and $D^{-s^{-1}}$ in this case. 
\end{proof}

\section{Necessary agreements and the main result}\label{s3}

In this section, we specify the restrictions on the used Turing machines,
the characteristics of their actions, and the methods of recording their 
instructions and Boolean formulae. These agreements are very important in   
proving the main theorem. Although any of these restrictions can be 
omitted at the cost of a complication of proofs.

\subsection{On the Turing machines and recording of Boolean formulae}\label{s3.1}
We reserve the following alphabet for the formulae of the signature of the 
two-element Boolean algebra $\mathcal{B}$: \\ a) signature symbols $\cap, 
\cup, C, 0, 1$ and equality sign $\approx$; \ \
b) Latin letters for the indication of the types of the object variables; \ \ 
c) Arabic numerals and comma for the writing of indices; \ \ 
d) Logical connectives $\neg, \wedge, \vee, \rightarrow$; \ \ e) the signs of 
quantifiers $\forall, \exists$; \ \ f) auxiliary symbols: (,). All these 
symbols constitute the first part of {\it a natural language}.%

\begin{rmk}\label{remeq} Let us pay attention to that we use three different  
symbols for the denotation of equality. The first is the signature symbol 
"$\approx$". It applies only inside the formulae of a logical theory. The second 
is the ordinary sign "$=$". It denotes the real or assumed equality and is used 
in our discussions on the formal logical system. The third sign 
"$\rightleftharpoons$" designates the equality in accordance with a definition.
\end{rmk} 

The priority of connectives and operations or its absence is inessential,
as a difference in length of formulae is linear in these cases.

Hereinafter we consider only deterministic machines with the fixed tape
alphabet $A$, which contains at least four symbols: the first of them is a 
designating "blank" symbol, denoted $\Lambda$; the second is a designating 
"start" symbol, denoted $\rhd$; and the last two are the numerals 0,1 (almost 
as in Section 1.2 of \cite{ArBar}). As usual, the machine cannot write or erase 
the $\rhd$ symbol. Nevertheless, the tape alphabet may only consist of two symbols 0 
and 1 as in some chapters of \cite{Mal}; at that the 0 symbol functions both as the 
$\rhd$ symbol and as the $\Lambda$ symbol. In this case, our modeling formulae 
become completely short (see Subsection \ref{s4.1} and Remark \ref{rem4}).

It is implied that the simulated machines have the only tape, seeing that the 
transformation of the machine program from a multi-tape variant to a 
single-tape version is feasible in the polynomial time on the length of the
program, at that the running time increases polynomially too 
\cite{AHU,ArBar,G-J}. Although the auxiliary machines may be multi-tape.

The machine tape is infinite only to the right, because the Turing machines 
are often considered in this manner (e.g., \cite{AHU,ArBar,ComHen,Ck,Fish-Rab,
G-J},\cite{Mal}--\cite{Vor}). 
Moreover, such machines can simulate the computations, which is $T$ steps in 
length on the two-sided tape machine, in linear time of $T$ \cite{ArBar}. %
The tape contains initially the start symbol $\rhd$ in the leftmost square, a 
finite non-blank input string $X$, and the blank symbol $\Lambda$ on the rest 
of its cells. The head is aimed at the left end of the tape, and the machine 
is in the special starting state $q_{start}\!=\!q_0$. When the machine recognizes 
an input, it enters the accepting state $q_{1}\!=\!q_{acc}$ or the rejecting 
state $q_{2}\!=\!q_{rej}$.

Our machines have the single-operand instructions of a kind \ 
$q_{i}\alpha\!\rightarrow \!q_{j}\beta$ as in \cite{Mal}, which differ from 
double-operand instructions of a form $q_{i}\alpha\!\rightarrow\!q_{j}\beta 
\gamma$, where $\alpha\!\in\!A; \ \beta,\gamma\!\in\!A\cup\!\{R,L\}$. Even if 
we regard the execution of a double-operand instruction as one step of 
computation, then the difference in length of the running time will be linear.

The Turing machines do not fall into a situation when the machine stopped, but 
its answer remained undefined. Namely, they do not try to go beyond the left 
edge of the tape; and besides, they do not contain the {\it hanging} (or {\it 
pending}) internal states $q_{j}$, for which $j\!\neq\!0,1,2$, and there exist 
instructions of a kind \ $\ldots\rightarrow q_{j}\beta$, but no instructions 
are beginning with $q_{j}\alpha\rightarrow\ldots$ at least for one 
$\alpha\!\in\!A$. The attempts to go beyond the left edge of the tape are 
blocked by the replacement of the instructions of a form $q_{i}\rhd\!\rightarrow
\!q_{k}L$ \ by \ $q_{i}\rhd\!\rightarrow\!q_{i}\rhd$. The hanging states are 
eliminated by adding the instructions of a kind $q_{j}\alpha\!\rightarrow\!
q_{j}\alpha$ for each of the missing alphabet symbol $\alpha$.

The programs of the single-tape Turing machines with the tape alphabet $A$ are 
written by the symbols of this alphabet and the application of the symbols 
$q,R,L,\rightarrow$, Arabic numerals, and comma. This is the second, 
last part of {\it a natural language}.

\subsection{The main theorem and its corollary}\label{s3.2}
Let $c_OM$ be a chosen polynomial code of an object $M$ by a string over a 
tape alphabet $A$ --- see the beginning of Subsection \ref{s2.2}. We suppose 
that for this encoding, there exists a linear function $l$ such that the 
inequalities $|M|\!\leqslant\!|c_OM|\!\leqslant\!l(|M|)$ hold for any object 
$M$ of the natural language described in the previous subsection. 

\begin{thm}\label{main}
For each deterministic Turing machine $P$ and every input string $X$, one can 
write a closed formula (sentence) $\Omega(X,P)$ of the signature of the 
two-element Boolean algebra $\mathcal{B}$ with the following properties:

(i)\ there exists a polynomial $g$ such that the code \ $c_O\,\Omega(X,P)$ 
\ is written within time \ $g(|X|,|c_OP|)$ for all $X$ and $P$;

(ii)\ $Th(\mathcal{B})\vdash\Omega(X,P)$ \ if and only if the Turing machine 
$P$ accepts input $X$ within time \ $\exp(|X|)$;

(iii)\ for every $\varepsilon\!>\!0$, there is a constant \ $D\!>\!0$ 
(depending on the used encoding) such that the inequalities \
$|X|\!<\!|c_O\,\Omega(X,P)|\!\leqslant\! D\cdot|c_OP|\cdot|X|^{2+\varepsilon}$
 \ hold for all sufficiently long $X$. \end{thm}

\begin{proof} See Sections \ref{s4}--\ref{s6}. Now we just note that according to 
the agreement in the beginning of this subsection, the calculation of the 
lengths of all components of the modeling formulae will be based on the estimate 
of the quantity of all the symbols, of the natural language of Subsection 
\ref{s3.1}, involved in their recording. 

At first, we will construct the very long formulae that simulate the 
computations. These formulae will have a huge number of "redundant" variables. 
We will take care of the brief record of the constructed formulae 
after we ascertain the correctness of our modeling (see Propositions \ref{Pr2}
(ii), \ref{Pr3}, and \ref{Pr4}(ii) below). The modified Stockmeyer and Meyer 
method is substantially used at that.\end{proof} 

\begin{cor}\label{comRecB} For every $\varepsilon\!>\!0$, \ $Th(\mathcal{B})\!
\notin\!DTIME(\exp(D^{-\rho}\cdot n^{\rho}))$, \ where $\rho\!=\!(2\!+\!
\varepsilon)^{-1}$\end{cor}

\begin{proof} It straightforwardly follows from the theorem and Corollary 
\ref{comRecL}. \end{proof} 

\begin{cor}\label{mainCor} The class $\mathbf{P}$ is a proper subclass of the 
class $\mathbf{PSPACE}$. \end{cor}

\begin{proof} Really, the theory $Th(\mathcal{B})$ does not belong to the class 
$\mathbf{P}$ under the previous corollary, and this theory is 
equivalent to the language $TQBF$ relative to polynomial reduction.  But the 
second language belongs to the class $\mathbf{PSPACE}$, moreover, it is 
polynomially complete for this class \cite{S-M}. \end{proof}

\begin{rmk}\label{rem3} This result is quite natural and expected for a long 
time. Its proof is yielded by one of the few possible ways. Indeed, since the 
language $TQBF$ is polynomially complete for the class $\mathbf{PSPASE}$, the 
inequality $\mathbf{P}\neq\mathbf{PSPASE}$ implies the impossibility of the 
inclusion $Th(\mathcal{B})\in\mathbf{P}$ that is almost equivalent to 
$Th(\mathcal{B})\notin DTIME(\exp(dn^{\delta}))$ for suitable $d,\delta\!>\!0$, 
as it is clear that $Th(\mathcal{B})\in DTIME(\exp(d_1n))$ for some $d_1$.\end{rmk} 

\subsection{Supplementary denotations and arrangements}\label{s3.3}
We introduce the following abbreviations and arrangements for the 
improvement in perception (recall that "$A \rightleftharpoons\mathcal{A}$"
means "$A$ is a designation for $\mathcal{A}$"):

(1) \ the square brackets and (curly) braces are equally applied with the 
ordinary parentheses in long formulae; \ \ (2) \ the connective $\wedge$ is 
sometimes written as $\&$; \ \ (3) \ the operation $\cap$ and connective \ 
$\wedge$ (\&) \ connect more closely than \ $\cup$ and \ $\vee, \rightarrow$; \ \ 
(4) \ $x\!<\!y\rightleftharpoons x\!\approx\!0\wedge y\!\approx\!1$; \ \
(5) \ $\langle\alpha_{0},\ldots,\alpha_{n}\rangle\!<\!\langle\beta_{0},\ldots,
\beta_{n}\rangle$ \  is the comparison of tuples in lexicographic ordering, 
i.e., it is the formula \
$$\alpha_{0}\!<\!\beta_{0}\vee\Bigl\{\alpha_{0}\!\approx\!\beta_{0}\wedge
\bigl[\alpha_{1}\!<\!\beta_{1}\vee\bigl(\alpha_{1}\!\approx\!\beta_{1}\wedge
\{\alpha_{2}\!<\!\beta_{2}\vee[\alpha_{2}\!\approx\!\beta_{2}
\wedge(\alpha_{3}\!<\!\beta_{3}\ldots)]\}\bigr)\bigr]\Bigr\}.$$

The symbol \ $\widehat{x}$ signifies an ordered set $\langle x_{0},\ldots,x_{n}
\rangle$, whose length is fixed. It is natural that "the formula" \ 
$\widehat{x}\!\approx\!\widehat{\alpha}$ denotes the system of equations \ 
$x_{0}\!\approx\!\alpha_{0}\!\wedge\ldots\wedge x_{n}\!\approx\!\alpha_{n}$.
The tuples of variables with two subscripts will occur only in the form where 
the first of these indices is fixed, for instance, $\langle u_{k,0},
\ldots,u_{k,n}\rangle$, and we will denote it by \ $\widehat{u}_{k}$.

Counting the length of a formula in the natural language, we are guided by the 
rule: a tuple $\widehat{x}$ has a length of $n\!+\!1$ plus $M$, which is the 
quantity of symbols involved in a record of the indices \ $0,\ldots,n$. The 
inequality \ $|\widehat{x}\!\approx\!\widehat{\alpha}|\!\leqslant\!M\!+\!3n\!+
\!3$ will hold, if $\widehat{\alpha}$ is a tuple of constants; and \
$|\widehat{x}\!\approx\!\widehat{\alpha}|\!\leqslant\!2M\!+\!3n\!+\!3$, 
when it consists of variables.

A binary representation of a natural number $t$ is denoted by $(t)_{2}$.

It is known that if \ $t\!=\!(\widehat{\gamma})\!=\!\langle\gamma_{0},\ldots,
\gamma_{n}\rangle_2$ \ is a binary representation of a natural number 
$t\!\leqslant\!\exp(2,n)$, then the numbers $t\!+\!1$ and $t\!-\!1$ will be 
expressed as \quad
$((\widehat{\gamma})\!+\!1)_2\!=\!\langle\gamma_{0}\!\oplus
\gamma_{1}\!\cdot\!...\!\cdot\gamma_{n-1}\!\cdot\gamma_{n},\,\ldots,\,
\gamma_{n-2}\!\oplus\gamma_{n-1}\!\cdot\gamma_{n},\,\gamma_{n-1}\!\oplus
\gamma_{n},\,\gamma_{n}\!\oplus\!1\rangle_2 $\\ \textrm{and} \qquad
$((\widehat{\gamma})\!-\!1)_2=\langle\gamma_{0}\!\oplus\!C\gamma_{1}\!\cdot\!
\ldots\!\cdot\!C\gamma_{n-1}\!\cdot\!C\gamma_{n},\,\ldots,\,\gamma_{n-2}\!\oplus
\!C\gamma_{n-1}\!\cdot\!C\gamma_{n},\,\gamma_{n-1}\!\oplus\!C\gamma_{n},
\,\gamma_{n}\!\oplus\!1\rangle_2,$ 
respectively, where the operation $\cap$ is written in the form of 
multiplication \ $x\cap y\!=\!x\cdot y$; and \ %
$x\oplus y\!\rightleftharpoons\!x\cdot Co(y)\cup Co(x)\cdot y$.

Let us pay attention that if $(\widehat{\gamma})$ is a binary representation 
of a natural number $t$, then $((t)_2)\!=\!t$ and $((\widehat{\gamma}))_2\!=\!
(t)_2$ according to our designations.

\begin{lem}\label{ll2} (i) $|\langle\alpha_{0},\ldots,\alpha_{n}\rangle\!<\!
\langle\beta_{0},\ldots,\beta_{n}\rangle|=\mathcal{O}(\max\{|\langle\alpha_{0},
\ldots,\alpha_{n}\rangle|,$\break $|\langle\beta_{0},\ldots,\beta_{n}\rangle|\})$.

(ii) If a tuple $(t)_{2}$ (together with the indices) is $l$ symbols in length, 
then the binary representation of the numbers $t\!\pm\!1$ will take up 
$\mathcal{O}(l^{2})$ symbols. \end{lem}

\begin{proof} It is obtained by direct calculation.\end{proof}

\section{The beginning of the proof of Theorem \ref{main}}\label{s4}

Prior to the writing of the formula $\Omega(X,P)$, we add $2|A|$ the {\em 
instructions of the idle run} to a program $P$, these instructions have the 
form \ $q_k\alpha\to q_k\alpha$, where \ $k\!\in\!\{1(accept),$ $2(reject)\}, 
\ \alpha\!\in\!A$. While the machine executes them, the tape configuration 
does not change.

\subsection{The primary and auxiliary variables}\label{s4.1}
In order to simulate the operations of a Turing machine $P$ on an input $X$ 
within the first \ $T\!=\!\exp(|X|)$ \ steps, it is enough to describe its  
actions on a zone, which is $T\!+\!1$ squares in width, since if $P$ 
starts its run in the zeroth cell, then it can finish a computation at most in 
the $T$th square. Because the record of the number $(T)_2$ has the $n\!+
\!1\!=\!|X|\!+\!1$ bit, the cell numbers are encoded by the values of the 
ordered sets of the variables of a kind "$x$": \ $\widehat{x}_{t}\!=\!\langle 
x_{t,0},\ldots, x_{t,n}\rangle$, which have a length of \ $n\!+\!1$. The first 
index $t$, i.e., {\it the color} of the record, denotes the step number, after 
which there appeared a configuration under study on the tape. So the formula \ 
$\widehat{x}_t\!\approx\!\widehat{\alpha}\rightleftharpoons x_{t,0}\!\approx\!
\alpha_{0}\wedge\ldots\wedge x_{t,n}\!\approx\!\alpha_{n}$ \ assigns the 
number \ $(\widehat{\alpha})$ \ of the required tape cell in the binary notation at 
the instant $t$.

Let us select so great a number $r$ in order that one can write down all the 
state numbers of the machine $P$ and encode all the symbols of the alphabet $A$ 
by means of the bit combinations of the same length $r\!+\!1$ at one time. 
Thus, $\exp(r\!+\!1)\!\geqslant\!|A|\!+\!U$, where $U$ is the maximal number 
of the internal states of $P$, and if $\beta\!\in\!A$, then 
$c\beta\rightleftharpoons\langle c\beta_{0},\ldots,c\beta_{r}\rangle$ will be 
the $(r\!+\!1)$-tuple, which codes $\beta$. So, the encoding $c_O$ applied 
in Sections \ref{s2} and \ref{s3} is "outside" (inherent a machine being 
simulated), and the encoding $c$ is "inner" (inherent a modeling formula). 

The formula $\widehat{f}_{t}\!\approx\!c\varepsilon$ represents an entry of 
symbol $\varepsilon$ in some cell after step $t$, where $\widehat{f}_{t}$ is 
the $(r\!+\!1)$-tuple of variables. When the cell, whose number is 
$(\widehat{\mu})$, contains the symbol 
$\varepsilon$ after step $t$, then this fact is associated with the {\it 
quasi-equation} (or {\it the clause}) of color $t$:
\begin{eqnarray*} \psi_{t}
(\widehat{\mu}\!\rightarrow\!\varepsilon) \ \rightleftharpoons \
[\widehat{x}_{t}\!\approx\!\widehat{\mu} \rightarrow
\widehat{f}_{t}\!\approx\!c\varepsilon] \ \rightleftharpoons \
[(x_{t,0}\!\approx\!\mu_{0}\wedge\ldots\wedge x_{t,n}\!\approx\!\mu_{n})
\!\rightarrow\\ \to (f_{t,0}\!\approx\!c\varepsilon_{0}\wedge\ldots\wedge
f_{t,r}\!\approx\!c\varepsilon_{r})]. \end{eqnarray*}

The tuples of variables $\widehat{q}_{t}$ and $\widehat{d}_{t}$ are 
accordingly used to indicate the number of the machine's internal state and 
the code of the symbol scanned by the head at the instant $t$. For every step 
$t$, a number \ $i\!=\!(\widehat{\delta})$ of the machine state $q_{i}$ 
and a scanned square's number $(\widehat{\xi})$ together with a symbol 
$\alpha$, which is contained there, are represented by a united
\emph{$\pi$-formula} of color $t$:  
\begin{eqnarray*}\pi_{t}(\alpha,(i)_{2}, \widehat{\xi})
\ \rightleftharpoons \ [\widehat{d}_{t}\!\approx\!c\alpha \wedge
\widehat{q_{t}}\!\approx\!\widehat{\delta}\wedge\widehat{z_{t}}\!
\approx\!\widehat{\xi}] \ \rightleftharpoons \ [(d_{t,0}\!\approx\!c
\alpha_{0}\wedge\ldots\wedge d_{t,r}\!\approx\!c\alpha_{r})\wedge\\ 
\wedge(q_{t,0}\!\approx\!\delta_{0}
\wedge\ldots\wedge q_{t,r}\!\approx\!\delta_{r}) 
\wedge(z_{t,0}\!\approx\!\xi_{0}\wedge\ldots\wedge
z_{t,n}\!\approx\!\xi_{n})],\end{eqnarray*}
where the ordered sets of variables $\widehat{d}_{t}$ and $\widehat{q}_{t}$ 
have a length of $r\!+\!1$; and $\widehat{z}_{t}$ is the $(n\!+\!1)$-tuple 
of variables and is assigned for the storage of the scanned cell's number. The 
formula expresses a condition for the applicability of instruction \ 
$q_i\alpha\!\rightarrow\!\ldots$; in other words, this is a {\it timer} that 
activates exactly this instruction, provided that the head scans the 
$(\widehat{\xi})$-th cell.

The basic variables \ $\widehat{x}_t, \widehat{z}_{t}$, and $\widehat{q}_{t}, 
\widehat{f}_{t}, \widehat{d}_{t}$ are introduced in great abundance in order 
to facilitate the proof. But a final modeling formula will only contain those 
of them for which $t\!=\!0,\ldots,n$ or \ $t\!=\!T \rightleftharpoons\exp(n)$ 
holds. The sets of the basic variables have the different lengths. However, 
this will not lead to confusion, since the tuples of the first two types \ 
$\widehat{x}_t$ and $\widehat{z}_{t}$ \ will always be $n\!+\!1$ in length, 
whereas the last ones \ $\widehat{q}_{t},\widehat{f}_{t}$, and $\widehat{d}_{t}$  
will have a length of $r\!+\!1$. The sets of constants or other variables may 
also be different in length, but such tuple will always be identically associated 
to some of the above mentioned ones.

The other variables are auxiliary. They will be described as needed. Their 
task consists of a determination of the values of the basic variables of the 
color \ $t+\!1$ provided that the primary ones of the color $t$ have the 
"correct" values. Moreover, this transfer must adequately correspond to that 
instruction which is employed at the step \ $t\!+\!1$. 

\begin{lem}\label{lcl} If the indices are left out of the account, then a clause 
$\psi_{t}(\widehat{u}\!\rightarrow\!\beta)$ and a timer ($\pi$-formula) will 
be \ $\mathcal{O}(n\!+\!r)$ in length. \end{lem}

\begin{proof} It is obtained by direct calculation. \end{proof}

\subsection{The description of an instruction action}\label{s4.2}
The following formula $\varphi(k)$ describes an action of the $k$th 
instruction \ $M(k)\!=\!q_{i}\alpha\!\rightarrow\!q_{j}\beta$ 
(including the idle run's instructions; see the beginning of this section) at 
some step, where $\alpha\!\in\!A$, \ $\beta\!\in\!A\cup\{R,L\}$:  
\begin{eqnarray*}\varphi(k) \ \rightleftharpoons \ \forall\,\widehat{u} \ 
\bigl\{\pi_{t}(\alpha,(i)_{2},\widehat{u}) \ \rightarrow \ 
\bigl[\Delta^{cop}(\widehat{u}(\beta)) \ \& \ \forall\widehat{h} \bigl(
\Gamma^{ret}(\beta) \ \rightarrow \\ \to\ [\Delta^{wr}(\beta)\ \&\ \pi_{t+1}
(h,(j)_2,\widehat{u}(\beta))]\bigr)\bigr]\bigr\}.\end{eqnarray*}
For the sake of concreteness, we regard that this step has a number $t\!+\!1$, 
so we have placed such subscripts on both $\pi$-formulae. Now we will describe the
sense of the subformulae of $\varphi(k)$ with the free basic variables \ 
$\widehat{x}_{t},\widehat{q}_{t},\widehat{z}_{t},\widehat{d}_{t},\widehat{f}_{t},
\widehat{x}_{t+1},\widehat{q}_{t+1},\widehat{z}_{t+1},$ $\widehat{d}_{t+1}$, and
$\widehat{f}_{t+1}$.%

The first $\pi$-formula of color $t$ plays a role of \emph{a timer}. It starts 
up the fulfillment of the instruction with the prefix \ 
$q_{i}\alpha\!\rightarrow\!\ldots$ provided that a head scans the  
$(\widehat{u})$th square. For every $(n\!+\!1)$-tuple $\widehat{u}$ \ and 
a given meta-symbol $\beta\!\in\!\{R,L\}\cup A$, the number of the cell that 
will be scanned by the head after the action of the instruction $M(k)$ 
is specified as follows: \quad
$\widehat{u}(R)\!\rightleftharpoons\!((\widehat{u})\!+\!1)_2;
\quad \widehat{u}(L)\!\rightleftharpoons\!((\widehat{u})\!-\!1)_2$; \ \ 
and \quad $\widehat{u}(\beta)\!\rightleftharpoons\!\widehat{u}$ \ \ for
$\beta\!\in\!A$.

The formula \ $\Delta^{cop}(\widehat{u}(\beta))$ changes the color of records 
in all the cells, whose numbers are different from $(\widehat{u}(\beta))$; 
in other words, it "copies" the majority of records (this subformula is universal 
in essence --- see Remark \ref{rem4}):
\begin{eqnarray*}\Delta^{cop}(\widehat{u}(\beta))\rightleftharpoons
\forall\,\widehat{w}\,[\neg(\widehat{w}\!\approx\!\widehat{u}(\beta))\rightarrow
\exists\,\widehat{g}(\psi_{t}(\widehat{w}\!\rightarrow\!\widehat{g}) \ \wedge 
\ \psi_{t+1}(\widehat{w}\!\rightarrow\!\widehat{g}))].\end{eqnarray*}

If \ $\beta\!\in\!\{R,L\}$, then \ $\Gamma^{ret}(\beta) \ \rightleftharpoons \  
\psi_{t}(\widehat{u}(\beta)\!\rightarrow\!\widehat{h}) \ = \ [\widehat{x}_{t}\!
\approx\!\widehat{u}(\beta)\rightarrow\widehat{f}_{t}\!\approx\!\widehat{h}]$. 
An informal sense of this formula is the following: it "seeks" a code 
$\widehat{h}$ of the symbol, which will be scanned after the next step $t\!+\!1$ 
(by this reason it is named "retrieval"); for this purpose, it "inspects" the 
square that is to the right or left of the cell $(\widehat{u})$. When \ 
$\beta\!\in\!A$, there is no need to look for anything, so the formula \ 
$\Gamma^{ret}(\beta)$ \ will be very simple in this case: \ 
$\widehat{h}\!\approx\!c\beta$.
 
The formula \ $\Delta^{wr}(\beta)$ "puts" the symbol, whose code is 
$\widehat{h}$ and color is $t\!+\!1$, in the $(\widehat{u}(\beta))$th 
square: \ $\Delta^{wr}(\beta)\rightleftharpoons\psi_{t+1}(\widehat{u}(\beta)\!
\rightarrow\!\widehat{h}))$. 

The second $\pi$-formula of the color $t+\!1$ aims the head at the  
$(\widehat{u}(\beta))$th cell; places the symbol $\widehat{h}$ in this 
location; and changes the number of the machine state for $j$: \ $\widehat{z}
_{t+1}\!\approx\!\widehat{u}(\beta)\wedge\widehat{d}_{t+1}\!\approx\!
\widehat{h}\wedge\widehat{q}_{t+1}\!\approx\!(j)_2$.

\begin{lem}\label{lphi} (i) If \ $\beta\!\in\!A$, then the formulae \ 
$\Gamma^{ret}(\beta)$; $\pi_{t+1}(\widehat{h},(j)_{2},\widehat{u}(\beta))$; 
$\Delta^{wr}(k,\beta)$; $\Delta^{cop}(\widehat{u},\beta)$; and $\varphi(k)$ 
will be $\mathcal{O}(|\psi_{t+1}(\widehat{w}\!\rightarrow\!\widehat{g})|)$ in 
length. 

(ii) For $\beta\!\in\!\{R,L\}$, each of these formulae is 
$\mathcal{O}(n\cdot|\psi_{t+1}(\widehat{w}\!\rightarrow\!\widehat{g})|)$ in  
length.\end{lem}

\begin{proof} This follows from Lemmata \ref{ll2} and \ref{lcl} by direct 
calculation. \end{proof}

\subsection{The description of the running steps and\\ configurations}\label{s4.3}
At first, we will construct a formula $\Phi^{(0)}(P)$ describing one step of 
the machine run, when the machine $P$ is applied to a configuration that arose 
after some step $t$. Next, we will describe the machine actions over an exponential 
period of time employing the formulae; at that, the Stockmeyer and Meyer 
method will be used.

\subsubsection{One step} \label{s4.3.1}
Let $N$ be a quantity of the instructions of the machine $P$ together with $2|A|$ 
the idle run's ones (see the beginning of this section). The formula 
$\Phi^{(0)}(P)$ that describes one step (whose number is $t\!+\!1$) of 
$P$ is of the form:
$$\Phi^{(0)}(P)(\widehat{y}_{t},\widehat{y}_{t+1}) \ \rightleftharpoons \
\bigwedge_{0<\,k\leq\,N}\varphi_{}(k)(\widehat{y}_{t},\widehat{y}_{t+1}),$$  
where \
$\widehat{y}_{t}\rightleftharpoons\langle\widehat{x}_{t},\widehat{q}_{t},
\widehat{z}_{t},\widehat{d}_{t},\widehat{f}_{t}\rangle$ \ \ and \ \
$\widehat{y}_{t+1}\rightleftharpoons\langle\widehat{x}_{t+1},\widehat{q}_{t
+1},\widehat{z}_{t+1},\widehat{d}_{t+1},\widehat{f}_{t+1}\rangle$ are two \ 
$(2n\!+\!3r\!+\!5)$-tuples of its free variables.\smallskip

\begin{rmk}\label{rem4}  Let $\langle\chi\rangle$ be a quantifier-free part of a  
formula $\chi$. Using the well-known Tarski and Kuratowski algorithm, we obtain 
that the prenex-normal form of $\varphi(k)$ is \ $\forall\,\widehat{u}\,\forall\,
\widehat{w}\,\forall\widehat{h}\,\exists\,\widehat{g}\langle\varphi(k)\rangle$. 
Recall that the variables $\widehat{u}$, $\widehat{w}$, $\widehat{g}$, 
and $\widehat{h}$ are auxiliary and "attached" to the corresponding basic 
variables, and so their lengths are the same as the primary ones (see Subsection 
4.1). The variables $\widehat{g}$ "service" only the variables $\widehat{f}_t$ and 
$\widehat{f}_{t+1}$. Therefore, the tuple $\widehat{g}$ has a length of $r\!+\!1$, 
but not $n\!+\!1$. Moreover, the destiny of the basic variables $\widehat{f}_t$ is 
a "storing" of information about the alphabet symbols, hence,  it is enough to 
consider the case, when the variables $\widehat{g}$ have the value of the 
alphabet's symbols codes. So we can substitute the subformula \ $\langle
\Delta^{cop}(\widehat{u}(\beta))\rangle$ \ with  $$\langle\Delta^{cop}
_1(\widehat{u}(\beta))\rangle\rightleftharpoons\neg(\widehat{w}\!\approx\!
\widehat{u}(\beta))\rightarrow\bigvee_{\delta\in A}(\psi_{t}(\widehat{w}\! 
\rightarrow\!\delta) \ \wedge \ \psi_{t+1}(\widehat{w}\!\rightarrow\!\delta)).$$
 It is clear that \ $|\Delta^{cop}_1(\widehat{u}(\beta))|<|A|\cdot|
\Delta^{cop}(\widehat{u}(\beta))|$. Thus, the formulae $\varphi(k)$ and $\Phi^{(0)}
(P)(\widehat{y}_{t},\widehat{y}_{t+1})$ \ are universal in essence.

But since this replacement does not result in the simplification of the proofs, we 
resume our work with the $\forall\exists$-formulae. \end{rmk}

\begin{lem}\label{trcl} (i) If \ $\widehat{x}_{t}\!\neq\!\widehat{\mu}$, then 
a clause $\psi_{t}(\widehat{\mu}\!\rightarrow\!\varepsilon)$ will be true 
independently of the value of variables $\widehat{f}_{t}$. In particular, a 
quasi-equation, which is contained into the record of 
$\langle\Delta^{cop}(\widehat{u})\rangle$ (this is quantifier-free part of 
$\Delta^{cop}(\widehat{u})$), will be true, if its color is $t$ or $t\!+\!1$, and 
at the same time \ $\widehat{x}_{t}\!\neq\!\widehat{w}$ or \ $\widehat{x}_{t+1}\!
\neq\!\widehat{w}$, respectively. 

(ii) For some constant $D_{1}$, the inequality \\
$|\Phi^{(0)}(P)(\widehat{y}_{t},\widehat{y}_{t+1})|\!\leqslant\!D_{1}\cdot
|c_OP|\cdot|\varphi_{}(N)|$ holds.\end{lem}

\begin{proof} (i) The premises of clauses are false in these cases.

(ii) If the quantity of the program $P$ instructions is not equal to zero, 
i.e., $N\!-\!2|A|\neq\!0$, then $N\cdot\lceil\lg N\rceil\!<\!D_{2}\cdot|c_OP|
$. This implies the assertion of the lemma. \end{proof}

\subsubsection{The configurations and the exponential quantity of steps}\label{s4.3.2}
The formulae $\Phi^{(s+1)}(P)(\widehat{y}_{t},\widehat{y}_{t+e(s+1)})$ 
conform to the actions of machine $P$ over a period of time \ 
$e(s)\!\rightleftharpoons\!\exp(s)$. They are defined by induction:
\begin{eqnarray*}\ \Phi^{(s+1)}(P)\rightleftharpoons\exists\,\widehat{v}\,
\forall\,\widehat{a}\,\forall\,\widehat{b}\;\bigl\{\bigl[(\widehat{y}_{t}\!
\approx\!\widehat{a} \ \wedge \ \widehat{v}\!\approx\!\widehat{b}) \ \vee \
(\widehat{v}\!\approx\!\widehat{a} \ \wedge \
\widehat{b}\!\approx\!\widehat{y}_{t+e(s+1)})\bigr] \to \\  
\rightarrow\Phi^{(s)}(P)(\widehat{a},\widehat{b})\bigr\},\end{eqnarray*} 
where \ $\widehat{v},\widehat{a},\widehat{b}$ are the $(2n\!+\!3r\!+\!5)$-tuples  
of the new auxiliary variables.

Let $L(t)$ be a configuration (maybe unrealizable), which is recorded on the tape 
after step $t$: namely, at the instant $t$, every cell, whose 
number is $(\widehat{\mu})$, contains a symbol $\varepsilon(\widehat{\mu})$; 
the scanned square has the number $(\widehat{\eta})$; and a machine is 
ready to execute an instruction \ $q_{i}\alpha\!\rightarrow\!\ldots$. Then the 
following formula corresponds  to this configuration (we recall that 
$T\!=\!\exp(n)$): \[ \Psi L(t)(\widehat{y}_{t}) \ \rightleftharpoons \
\pi_{t}(\alpha,(i)_{2},\widehat{\eta}) \ \& \
\bigwedge_{0\leqslant\,(\widehat{\mu})_2\leqslant\,T}
\psi_{t}(\widehat{\mu}\!\rightarrow\!\varepsilon(\widehat{\mu})).\]
It has \ $2n\!+\!3r\!+\!5$ \ free variables \
$\widehat{y_{t}}\!=\!\langle\widehat{x}_{t},\widehat{q}_{t},$
$\widehat{z}_{t},d_{t},f_{t}\rangle$.

\section{The simulation of one running step}\label{s5}

We simply associated the formulae, which were constructed earlier, with the 
certain components of programs or with processes. However one cannot assert 
that these formulae simulate something, i.e., they will not always turn true, 
when the events, which are described by them, are real.%

\subsection{Simulating formula}\label{s5.1}
Let us define 
$$\Omega^{(0)}(X,P)(\widehat{y}_{t},\widehat{y}_{t+1})
\rightleftharpoons[\Psi K(t)(\widehat{y}_{t}) \ \& \
\Phi^{(0)}(P)(\widehat{y}_{t},\widehat{y}_{t+1})] \rightarrow \Psi
K(t\!+\!1)(\widehat{y}_{t+1}).$$

We will prove in this section that the sentence $\forall\widehat{y}_t\forall
\widehat{y}_{t+1}\Omega^{(0)}(X,P)(\widehat{y}_{t},\widehat{y}_{t+1})$ is true 
on the Boolean algebra $\mathcal{B}$ if and only if the machine $P$ transforms 
the configuration $K(t)$ into $K(t\!+\!1)$ in one step. So we can say that this 
formula models the machine actions at the step $t\!+\!1$.%

\begin{rmk}\label{rem5}{\rm One can regard that the formula $\Omega^{(0)}(X,P)
(\widehat{y}_{t},\widehat{y}_{t+1})$ is the analog of the Cook's method formula 
$A_{0,m}(\widetilde{U},\widetilde{V})$, which was applied in the proof of 
Theorem 4.3 in \cite{S-M}, here $\widetilde{U}$ and $\widetilde{V}$ are the 
sequences $(u_1,\ldots,u_m)$ and $(v_1,\ldots,v_m)$ of the Boolean variables and 
$m\!=\!q(|X|)$ is the value of suitable polynomial $q$ on the length of input 
$X$. Indeed, the sentence $\exists\widetilde{U}\exists\widetilde{V}A_{0,m}
(\widetilde{U},\widetilde{V})$ is true if and only if the configuration encoded 
by $v_1\ldots v_m$ follows from the configuration that corresponds to $u_1\ldots 
u_m$ in at most one step of the $P$ (these $m$ and $P$ are $n$ and $\mathfrak{M}
$ in \cite{S-M}). 

However, there are solid arguments to believe that the real analog of the 
formula $A_{0,m}(\widetilde{U},\widetilde{V})$ is the $\Phi^{(0)}(P)(\widehat{y}
_{t},\widehat{y}_{t+e(s+1)})$ nevertheless. We will return to the discussion of 
this analogy in Subsection \ref{s8.2}. }\end{rmk}

\subsection{The single-valuedness of modeling and\\ the special values of variables}\label{s5.2}
Let \ $K(t\!+\!1)$ be a configuration that has arisen from a configuration 
$K(t)$ as a result of the machine $P$ action at the step $t\!+\!1$.

\begin{prop}\label{Pr2}
(i) There exist special values of variables $\widehat{y}_{t}$ such that the 
formula \ $\Psi K(t)(\widehat{y}_{t})$ is true, and the truth of \ 
$\Phi^{(0)}(P)(\widehat{y}_{t},\widehat{y}_{t+1})$ follows from the truth of \ 
$\Psi K(t\!+\!1)(\widehat{y}_{t+1})$ for every $\widehat{y}_{t+1}$.

(ii) If a formula $\Omega^{(0)}(X,P)(\widehat{y}_{t},\widehat{y}_{t+1})$ is 
identically true over algebra $\mathcal{B}$, then the machine $P$ cannot 
convert the configuration $K(t)$ into the configuration, which differs from 
$K(t\!+\!1)$, at the step $t+\!1$. \end{prop}

\begin{proof} We will prove these assertions simultaneously. Namely, we will 
select the values variables $\widehat{y}_{t}$ and $\widehat{y}_{t+1}$ such that a 
formula 
$$\Upsilon_{t+1}(\widehat{y}_{t},\widehat{y}_{t+1})\! \rightleftharpoons\!
[\Psi K(t)(\widehat{y}_{t})\,\&\,\Phi^{(0)}(P)(\widehat{y}_{t},\widehat{y}_{t
+1})]\!\rightarrow\!\Psi L(t+\!1)(\widehat{y}_{t+1})$$ will be false, if the
configuration \ $L(t+\!1)$ differs from the real $K(t+\!1)$. This implies 
Item (ii) of the proposition. However, at the beginning, we will select the 
special values of the variables of the tuple $\widehat{y}_{t}$. After 
that when we pick out the values of the corresponding variables of the color 
$t+\!1$, the formulae \ $\Psi K(t+\!1)(\widehat{y}_{t+1})$  and \ 
$\Phi^{(0)}(P)(\widehat{y}_{t},\widehat{y}_{t+1})$ will become true or false at 
the same time depending on the values of the variables $\widehat{y}_{t+1}$. 

Let \ $M(k)\!=\!q_{i}\alpha\!\rightarrow\ldots$ be an instruction that is
applicable to the configuration $K(t)$; and $(\widehat{\eta})$ be a number 
of the scanned square. We specify \ $\widehat{d}_{t}\!=\!c\alpha,\
\widehat{q}_{t}\!=\!(i)_{2},\ \widehat{z}_{t}\!=\!\widehat{\eta}$. Then
$\pi$-formula $\pi_{t}(\alpha,(i)_{2},\widehat{\eta})$, which is in the record  
of $\Psi K(t)$, is true. 

Let us consider a formula $\varphi(l)$ that conforms to some instruction 
$M(l)\!=\!q_{b}\theta\!\rightarrow\ldots$ that differs from $M(k)$. 
This formula has a timer \ $\pi_{t}(\theta,(b)_{2},\widehat{u})$ as the first 
premise. For the selected values of the variables $\widehat{d}_{t}; 
\widehat{q}_{t};$ and $\widehat{z}_{t}$, the timer takes the form of \ 
$c\alpha\!\approx\!c\theta\wedge(i)_{2}\!\approx\!(b)_{2}\wedge
\widehat{\eta}\!\approx\!\widehat{u}$. It is obvious that if \
$\alpha\!\neq\!\theta$; or \ $i\!\neq b$; or \
$\widehat{u}\!\neq\!\widehat{\eta}$, then this $\pi$-formula will be false, 
and the whole $\varphi(l)$ will be true.

Thus, let $\varphi(k)$ be a formula that correspondents to the instruction \ 
$M(k)\!=\!q_{i}\alpha\!\rightarrow\!q_{j}\beta$; and
$\widehat{u}\!=\!\widehat{\eta}$. Let us define \
$\widehat{d}_{t+1}\!=\!c\lambda; \ \widehat{q}_{t+1}\!=\!(j)_{2};
\ \widehat{z}_{t+1}\!=\!\widehat{\eta}(\beta)$, where
$(\widehat{\eta}(\beta))$ is a number of the square, which will be 
scanned by the machine head after the fulfillment of the instruction $M(k)$; 
and $\lambda$ is the symbol, which the head will see there. For these
$\widehat{u}$ and selected values of \ $\widehat{d}_{t+1},\widehat{q}_{t+1}, 
\widehat{z}_{t+1}$, \ the $\pi$-formula, which enters into the record of 
$\Psi K(t+\!1)$, becomes true. But the conclusion of the quantifier-free part 
$\langle\varphi(k)\rangle$ contains a slightly different timer 
$\pi_{t+1}(\widehat{h},(j)_{2},\widehat{u}(\beta))$; in this timer, the only 
equality\ $\widehat{d}_{t+1}\!\approx\!\widehat{h}$ included in it raises
doubts for the time being.%

Let us assign \ $\widehat{x}_{t}\!=\!\widehat{\eta}(\beta)$. Since we consider  
the case, when $\widehat{u}\!=\!\widehat{\eta}$, the equality $\widehat{u}
(\beta)\!=\!\widehat{\eta}(\beta)$ holds too. Therefore the quasi-equation, of the 
color $t$, which enters into $\langle\Delta^{cop}(\widehat{u}(\beta))\rangle$ 
(this is the quantifier-free part of $\Delta^{cop}(\widehat{u}(\beta))$), is true 
for all $\widehat{w}\!\neq\!\widehat{\eta}(\beta)$ and irrespective of the values 
of the tuples $\widehat{f}_{t}$ and $\widehat{g}$ according to Lemma \ref{trcl}(i). 
For the same reason, all the clauses that are included in $\Psi K(t)$ are true, 
except \ $\psi_{t}(\widehat{\eta}(\beta)\rightarrow\!\lambda)$ for 
$\beta\!\in\!\{R,L\}$ or $\psi_{t}(\widehat{\eta}\!\rightarrow\!\alpha)$ for 
$\beta\!\in\!A$. We set the value of the tuple $\widehat{f}_{t}$ as 
$c\lambda$, if $\beta\!\in\!\{R,L\}$, or as $c\alpha$, if not. Now, the 
questionable clause from $\Psi K(t)$ becomes true, because its premise and 
conclusion are true.%

If \ $\widehat{h}\!\neq\!c\lambda$, then the formula \ $\Gamma^{ret}(\beta)$ 
will be false, since it is either \
$\psi_{t}(\widehat{\eta}(\beta)\!\rightarrow\!\widehat{h})$ for \
$\beta\!\in\!\{R,L\}$, or $\widehat{h}\!\approx\!c\beta$ for \ $\beta\!\in
\!A$. Hence the whole formula $\langle\varphi(k)\rangle$ will be true in this 
case. When $\widehat{h}\!=\!c\lambda$, the terminal $\pi$-formula in 
$\langle\varphi(k)\rangle$ becomes true, as $\widehat{d}_{t+1}\!=\!c\lambda$.

If the "incorrect" formula $\Psi L(t+\!1)$ has a mistake in the record of 
timer or clause $\psi_{t+1}(\widehat{\eta}(\beta)\rightarrow\!\lambda)$, we 
will define \ $\widehat{x}_{t+1}\!=\!\widehat{\eta}(\beta)$ \ and \ 
$\widehat{f}_{t+1}\!=\!c\lambda$ (we recall that $\lambda\!=\!\beta$ for 
$\beta\!\in\!A$). But when these fragments are that as they should be, however, 
there is another "incorrect" clause $\psi_{t+1}(\widehat{\mu}\!\rightarrow\!
\rho)$, where $\rho$ is different from "real" $\delta$, we will assign 
$\widehat{x}_{t+1}\!=\!\widehat{\mu}$ and \ $\widehat{f}_{t+1}\!=\!\widehat{g}\!
=\!c\delta$.\footnote{We note that this is the only case when we need to set 
the values of the variables $\widehat{g}$.} We obtain again that all the 
quasi-equations of the color $t\!+\!1$ in the formulae $\langle
\Delta^{cop}(\widehat{u}(\beta))\rangle$ and $\Delta^{wr}(\beta)$ are true in 
both of these cases on the grounds of Lemma \ref{trcl}(i) or because their 
premises and conclusions are true. Therefore the whole formula $\langle
\varphi(k)\rangle$ is true. All the clauses contained in $\Psi K(t\!+\!1)$ are 
true for the same reasons.

We obtain as a result that any formula $\varphi(l)$ is true for the above 
selected values of the primary variables, so the entire conjunction 
$\Phi^{(0)}(P)$ is true. Since the premise and conclusion of \
$\Omega^{(0)}(X,P)(\widehat{y}_{t},\widehat{y}_{t+1})$ are true, and the
configurations $K(t+\!1)$ and $L(t+\!1)$ are different; the "incorrect"
formula $\Upsilon_{t+1}$ is false. 

In view of the fact that the configuration $L(t\!+\!1)$ may differ from the real 
$K(t+\!1)$ in any place, Item (i) of Proposition is established too. \end{proof}

\subsection{The sufficiency of modeling}\label{s5.3}
We will now prove a converse to Proposition \ref{Pr2}(ii).

\begin{prop}\label{Pr3} Let \ $K(t\!+\!1)$ be a configuration that has arisen
from a configuration $K(t)$ as a result of an action of the machine $P$ at 
the step $t\!+\!1$. Then the formula
$\Omega^{(0)}(X,P)(\widehat{y}_{t},\widehat{y}_{t+1})$ is identically true 
on algebra $\mathcal{B}$. \end{prop}

\begin{proof} Let \ $M(k)\!=\!q_{i}\alpha\!\rightarrow\!q_{j}\beta$ be the 
instruction that transforms the configuration $K(t)$ into $K(t\!+\!1)$; 
and $\varphi(k)(\widehat{y}_{t},\widehat{y}_{t+1})$ be a formula, which is 
written for this instruction. This formula is the consequence of \ 
$\Phi^{(0)}(P)(\widehat{y}_{t},\widehat{y}_{t+1})$.

Let us replace $\varphi(k)$ by a conjunction of formulae
$\varphi(k)(\widehat{\mu})$, they are each obtained as the result of the
substitution of the various values of the universal variables $\widehat{u}$ in place 
of the variables themselves. Every formula $\varphi(k)(\widehat{\mu})$ contains 
the premise \
$\widehat{d}_{t}\!\approx\!c\alpha\wedge\widehat{q_{t}}\!\approx\!(i)_{2}
\wedge\widehat{z}_{t}\!\approx\!\widehat{\mu}$, one of them coincides with
the only timer \ $\pi_{t}(\alpha,\widehat{\delta},\widehat{\eta})$ included 
in $\Psi K(t)$ for \ $\widehat{u}\!=\!\widehat{\mu}\!=\!\widehat{\eta}$ 
and \ $i\!=\!(\widehat{\delta})$, as the instruction $M(k)$ is applicable 
to the configuration $K(t)$. Therefore the formula \quad 
$\Psi K(t) \ \& \ \Delta^{cop}(\widehat{\eta}(\beta)) \ \& \ 
\forall\,\widehat{h}\{\Gamma^{ret}(\beta)(\widehat{\eta})
\rightarrow[\Delta^{wr}(\beta)(\widehat{\eta}) \ \& \
\pi_{t+1}((\widehat{h},(j)_2,$ $\widehat{\eta}(\beta))]\}$  follows from  
$\Psi K(t)$ and $\varphi(k)(\widehat{\eta})$.

The formula $\Delta^{cop}(\widehat{\eta}(\beta))$ begins with the quantifiers 
$\forall\,\widehat{w}$. Let us replace this formula wits a conjunction that is  
equivalent to it, we substitute all possible values for the variables 
$\widehat{w}$ to this effect. For every value of $\widehat{w}$, there is a 
unique value of the tuple $\widehat{g}$ such that the clause \ 
$\psi_{t}(\widehat{w}\!\rightarrow\!\widehat{g}_{k})$ enters into the 
formula $\Psi K(t)$. When these values of $\widehat{g}$ are substituted 
in their places, we will obtain all the quasi-equations of $\Psi K(t\!+\!1)$, 
except one.

For the appropriate value of $\widehat{h}$, either the formula 
$\Gamma^{ret}(\beta)(\widehat{\eta})$ coincides with some clause existing 
in $\Psi K(t)$, or it becomes true: $\widehat{h}\!\approx\!c\beta$, owing to 
the instruction $M(k)$ applicabi\-lity to the configuration $K(t)$. In any 
case, the formula \ $\Delta^{wr}(\beta)(\widehat{\eta})$ in an explicit form 
contains the quasi-equation \
$\psi_{t+1}(\widehat{\eta}(\beta)\!\rightarrow\ldots)$, which is missing in  \ 
$\Psi K(t\!+\!1)$ so far; and the tuple $\widehat{h}$ obtains the concrete value. 
If we substitute this value in the concluding $\pi$-formula of $\varphi(k)$, then 
we will obtain the necessary timer \
$\pi_{t+1}(\widehat{h},(j)_2,\widehat{\eta}(\beta))$ \ in 
$\Psi K(t\!+\!1)$. \end{proof} 

\section{The construction of the formula $\Omega(X,P)$}\label{s6}

\subsection{The simulation of the exponential computations}\label{s6.1} 
Let us define the formulae \
$\Omega^{(s)}(X,P)(\widehat{y}_{t},\widehat{y}_{t+e(s)})$ that model 
$e(s)\!\rightleftharpoons\!\exp(s)$ running steps of a machine $P$, when 
it applies to a configuration $K(t)$:
\begin{eqnarray*}\Omega^{(s)}(X,P)(\widehat{y}_{t},\widehat{y}_{t+e(s)}) \ 
\rightleftharpoons\ [\Psi K(t)(\widehat{y}_{t}) \ \& \ \Phi^{(s)}(P)
(\widehat{y}_{t},\widehat{y}_{t+e(s)})] \ \to\\ \to\Psi K(t\!+\!e(s))(\widehat{y}
_{t+e(s)}).\end{eqnarray*}

\begin{prop}\label{Pr4} Let $t,s\!\geqslant\!0$ be the integers such that 
\ $t\!+\!e(s)\!\leqslant\!T$.%

(i) If the machine $P$ transforms the configuration $K(t)$ into $K(t\!+\!e(s))$ 
within $e(s)$ steps, then there are special values of variables $\widehat{y}_{t}$ 
such that the formula \ $\Psi K(t)(\widehat{y}_{t})$ \ is true; and for all 
$\widehat{y}_{t+e(s)}$, whenever the formula\break \ $\Psi K(t\!+\!e(s))
(\widehat{y}_{t+e(s)})$ is true, \ $\Phi^{(s)}(P)(\widehat{y}_{t},\widehat{y}_{t
+e(s)})$ is also true. %

(ii) The formula \
$\Omega^{(s)}(X,P)(\widehat{y}_{t},\widehat{y}_{t+e(s)})$ is identically 
true over the Boolean algebra $\mathcal{B}$ if and only if the machine $P$ 
converts the configuration $K(t)$ into $K(t\!+\!e(s))$ within $e(s)$ steps. 
\end{prop}%

\begin{proof}
Induction on the parameter $s$. For $s\!=\!0$, Item (i) is a consequence of 
Proposition~\ref{Pr2}(i), and Item (ii) follows from Propositions~\ref{Pr2}(ii) 
and~\ref{Pr3}.

We start the proof of the inductive step by rewriting the formula \break
$\Phi^{(s+1)}(P)(\widehat{y}_{t},$ $\widehat{y}_{t+e(s+1)})$ in the equivalent, 
but longer form:
\begin{eqnarray*}\exists\,\widehat{v}\bigl\{\forall\,\widehat{a}\,\forall\,
\widehat{b} \bigl[(\widehat{y}_{t}\!\approx\!\widehat{a} \ \wedge \
\widehat{v}\!\approx\!\widehat{b}) \ \rightarrow \ \Phi^{(s)}(P)(\widehat{a},
\widehat{b}) \bigr] \ \& \\ \& \ \forall\,\widehat{a}\,\forall\,\widehat{b}
\bigl[(\widehat{v}\!\approx\!\widehat{a} \ \wedge \
\widehat{b}\!\approx\!\widehat{y}_{t+e(s+1)}) \ \rightarrow \ \
\Phi^{(s)}(P)(\widehat{a},\widehat{b})\bigr]\bigr\}.\end{eqnarray*}
The following formula results from this immediately:
$$\Xi_{s+1}\rightleftharpoons\exists\,\widehat{v}\bigl\{\Phi^{(s)}(P)
(\widehat{y}_{t},\widehat{v})
\ \&\Phi^{(s)}(P)(\widehat{v},\widehat{y}_{t+e(s+1)})\bigr\}.$$
 On the other hand, each of the two implications which are included in the 
equi\-va\-lent long form of the formula \
$\Phi^{(s+1)}(P)(\widehat{y}_{t},\widehat{y}_{t+e(s+1)})$ can be false 
only when the equalities existing in its premise are valid. Hence this long 
formula is equivalent to $\Xi_{s+1}$.

Let the machine $P$ transforms the configuration $K(t)$ into $K(t+\!e(s))$ within 
$e(s)$ steps, and it converts the latter into $K(t\!+\!e(s\!+\!1))$ within the 
same time.

By the inductive hypothesis of Item (ii) (we recall that the induction is  
carried out over a single parameter $s$), the formula \ $\Omega^{(s)}(P)
(\widehat{y}_{t},\widehat{y}_{t+e(s)})$ is identically true \ for any $t$ such
that $t\!+\!e(s)\!\leqslant\!T$, and hence it is identically true for an 
arbitrarily chosen $t$ and for \ $t_1\!=\!t\!+\!e(s)$ provided that \
$t\!+\!e(s\!+\!1)\!=\!t_1\!+\!e(s)\!\leqslant\!T$. Thus, the formulae
\begin{eqnarray*} [\Psi K(t)(\widehat{y}_{t}) \
\& \ \Phi^{(s)}(P)(\widehat{y}_{t},\widehat{y}_{t+\!e(s)})]\!\rightarrow\!\Psi 
K(t\!+\!e(s))(\widehat{y}_{t+e(s)})\ \quad \textrm{and} \\ 
\{\Psi K(t\!+\!e(s))(\widehat{y}_{t+e(s)}) \& \Phi^{(s)}(P)(\widehat{y}_{t
+e(s)},\widehat{y}_{t+e(s+1)})\}\!\rightarrow\!\Psi K(t\!+\!e(s\!+\!1))
(\widehat{y}_{t+e(s+1)}) \end{eqnarray*}
are identically true. Therefore, when we change the variables under the sign 
of the quantifier, we obtain from this that the following formula   
\begin{eqnarray*}\forall\widehat{v}\{[\Psi K(t)(\widehat{y}_{t}) \ \& \ \Phi^{(s)}
(P)(\widehat{y}_{t},\widehat{v}) \ \& \ \Phi^{(s)}(P)(\widehat{v},
\widehat{y}_{t+e(s+1)})]\rightarrow\\ \to \ \Psi K(t\!+\!e(s\!+\!1))(\widehat{y}
_{t+e(s+1)})\},\end{eqnarray*}   
is identically true as well. This formula is equivalent to \
$[(\Psi K(t) \ \& \ \Xi_{s+1})\!\rightarrow\!\Psi K(t\!+\!e(s\!+\!1))](\widehat{y}
_{t},\widehat{y}_{t+e(s+1)})$, because the universal quantifiers  $\forall
\widehat{v}$ will be interchanged with the ones of existence  $\exists\widehat{v}$, 
when they are introduced into the premise of the implication. Since 
the premise of the formula \break $\Omega^{(s+1)}(X,P)(\widehat{y}_{t},\widehat{y}
_{t+e(s+1)})$  is equivalent to $\Psi K(t) \ \& \ \Xi_{s+1}$ in accordance with the 
foregoing argument, the inductive step of Item~(ii) is proven in one direction.%

Now let the configurations \ $L(t+\!e(s+\!1))$ and \ $K(t+\!e(s+\!1))$ be 
different. For some \ $\widehat{v}_{1},\widehat{y}_{t+e(s+1)}$, the formula 
\ $$\{[\Psi K(t\!+\!e(s)) \ \& \ \Phi^{(s)}(P)]\rightarrow\Psi
K(t+\!e(s+\!1))\}(\widehat{v}_{1},\widehat{y}_{t+e(s+1)})$$ \ is true, but \ 
$$\{[\Psi K(t\!+\!e(s)) \ \& \ \Phi^{(s)}(P)]\rightarrow
\Psi L(t+e(s+\!1))\}(\widehat{v}_{1},\widehat{y}_{t+e(s+1)})$$ \ 
is false by the inductive assumption of Item (ii). Therefore, the 
conclusion of the se\-cond formula is false, and its premise is true, i.e., \ 
$\Psi L(t+e(s\!+\!1))(\widehat{y}_{t+e(s+1)})$ is false, but \
$\Phi^{(s)}(P)(\widehat{v}_{1},\widehat{y}_{t+e(s+1)})$ \ and \ 
$\Psi K(t+e(s))(\widehat{v}_{1})$ are true. Since the last formula and 
$\Psi K(t)(\widehat{v}_{0})$ are true for some special $\widehat{v}_{0}$, 
which exists due to induction proposition of Item~(i), the formula \
$\Phi^{(s)}(P)(\widehat{v}_{0},\widehat{v}_{1})$ is true. Thus, the 
implication \ $$\{[\Psi K(t) \ \& \ \Phi^{(s+1)}(P)]\rightarrow\Psi
L(t\!+\!e(s\!+\!1))\}(\widehat{v}_{0},\widehat{y}_{t+e(s+1)})$$
\ has a true premise, and a false conclusion, therefore it is not 
identically true. Item~(ii) is proven.

Inasmuch as the configuration $L(t\!+\!e(s+\!1))$ may differ from the current 
one at any position, to finish the proof of Item~(i) we set the values of 
the variables $\widehat{v}_{1}$ in a special manner, using the inductive 
hypothesis. \end{proof}%

\subsection{The short recording of the initial configuration and\\ the condition of the successful termination of the machine run} \label{s6.2}
Since we have the instructions for the machine run in the idle mode (see the 
beginning of Section~\ref{s4}), the statement that the machine $P$ accepts an 
input string $X$ within \ $T\!=\!\exp(n)$ steps can be written rather brief 
--- by means of one quantifier-free formula of the color $T$: \
$\chi(\omega)\rightleftharpoons\widehat{q}\,_{T}\!\approx\!(1)_2$.
This formula has a length of $4r\!+\!3$ symbols nonmetering the indices. The 
writing of the first index $T$ occupies $\lfloor\lg T\rfloor\!+\!1$ digits 
(indices are written in decimal notation, not in binary), where \ 
$\lg m\!=\!\log_{10}m$, $\lfloor y\rfloor$ is the integer part of a number $y$. 
The maximum length of the second indices is $\lfloor\lg r\rfloor\!+\!1$, and so 
we have \ $|\chi(\omega)|\!<\!(4r\!+\!3)\cdot
(\lfloor\lg T\rfloor\!+\!\lfloor\lg r\rfloor\!+\!2)$.

The formula $\Psi K(t)$ was introduced in Subsection~\ref{s4.3.2} to describe 
a confi\-gu\-ration arising after the step $t$. It is very long --- much 
longer than $n\cdot\exp(n)$. However, the initial configuration consists of 
the input string $X$, which occupies the $|X|$ squares to the right of the 
edge of a tape; the head points to this extreme left cell; and the remaining 
part of the tape is empty, starting with the cell, whose number is \ 
$|X|\!+\!1\!=\!(\widehat{\gamma})\!+\!1$. Therefore one can describe the 
initial tape configuration $K(0)$ by a brief universal formula:
$$\chi(0)\rightleftharpoons\pi_{0}(\rhd,\widehat{0},\widehat{0}) \ \&
\bigwedge_{0\leqslant(\widehat{\eta})_{2}\!\leqslant|X|}\psi_{0}
(\widehat{\eta}\!\rightarrow\!\alpha(\eta)) \ \& \ \forall\widehat{u}_{0}
[(\widehat{u}_{0}\!>\!\widehat{\gamma})\rightarrow
\psi_{0}(\widehat{u}_{0}\!\rightarrow\!\Lambda)],$$
 where $\Lambda$ denotes blank symbol; $\rhd$ is a sign of the left end of the 
tape; and $\alpha(\eta)$ is a symbol, which is located in the number \ 
$(\widehat{\eta})$ cell; and the $\pi$-formula of the color 0 signifies 
that a mechanism is ready for the execution of the instruction \ 
$q_{0}\rhd\!\rightarrow\ldots$ at the zeroth instant, and the machine head is
positioned on the extreme left square of tape and scans $\rhd$ symbol.

\begin{lem}\label{lICon} (i) The formulae $\chi(0)$ and $\Psi K(0)$ are 
equivalent to each other.

(ii) $|\chi(0)(\widehat{y}_{0})|\!\leq\!D_{2}\cdot|X|\cdot
|\psi_{0}(\widehat{u}_{0}\!\rightarrow\!\Lambda)|$ for a proper constant 
$D_{2}$. \end{lem}

\begin{proof} (i) The first, quantifier-free part of the formula $\chi(0)$ simply 
coincides with the initial fragment of the formula $\Psi K(0)$. If we replace 
the second part of formula $\chi(0)$, which begins with the quantifiers 
$\forall\widehat{u}_{0}$, with its equally matched conjunction, the rest of 
the clauses from $\Psi K(0)$ will appear.

(ii) According to Lemmata \ref{ll2}(i) and \ref{lcl}, the system of 
inequalities $\widehat{u}_{0}\!>\!\widehat{\gamma}$ has a length of the same 
order as $|\psi_{0}(\widehat{u}_{0}\!\rightarrow\!\Lambda)|$, a quantifier 
prefix is a bit shorter. Hence \ $|\forall\widehat{u}_{0}
[\widehat{u}_{0}\!>\!\widehat{\gamma}\rightarrow\psi_{0}(\widehat{u}_{0}
\!\rightarrow\!\Lambda)]|\!=\!\mathcal{O}(
|\psi_{0}(\widehat{u}_{0}\!\rightarrow\!\Lambda)|)$. Since the expression 
$\chi(0)(\widehat{y}_{0})$ includes $|X|\!+\!1$ quasi-equations of a form 
$\psi_{0} (\widehat{\eta}\!\rightarrow\!\alpha(\eta))$ and the timer, 
which have a length of the same order as \ 
$|\psi_{0}(\widehat{u}_{0}\!\rightarrow\!\Lambda)|$ by Lemma \ref{lcl}, the 
whole formula $\chi(0)(\widehat{y}_{0})$ is not more than
$D_{2}\cdot|X|\cdot|\psi_{0}(\widehat{u}_{0}\!\rightarrow\!\Lambda)|$ in
length for some constant $D_{2}$. \end{proof}

\subsection{Simulating formula $\Omega(X,P)$}\label{s6.3}
Let us define
\begin{eqnarray}
 \Omega(X,P) \quad \rightleftharpoons
\quad \forall\widehat{y}_{0},\widehat{y}\,_{T} \ \Bigl\{ \ \Bigl[
\ \chi(0)(\widehat{y}_{0}) \quad \& \
\exists\,\widehat{v}_{n}\forall\,\widehat{a}_{n}\forall\,\widehat{b}_{n}
\ldots\exists\,\widehat{v}_1\forall\,\widehat{a}_{1}\forall\,
\widehat{b}_{1} \nonumber \hphantom{a}\\
\bigl\{\bigwedge\limits_{1\leqslant\,s\leqslant\,n}\bigl[
(\widehat{a}_{s+1}\!\approx\!\widehat{a}_{s}\wedge
\widehat{v}_{s}\!\approx\!\widehat{b}_{s}) \ \vee
(\widehat{v}_{s}\!\approx\!\widehat{a}_{s}\wedge\widehat{b}_{s}\!
\approx\!\widehat{b}_{s+1})\bigr]  \hphantom{aaaaaaa} \\
\rightarrow \
\Phi^{(0)}(P)(\widehat{a}_{1},\widehat{b}_{1})\bigr\}\Bigr] \ \rightarrow 
  \ \chi(\omega)(\widehat{y}\,_{T})\Bigr\},
\nonumber \hphantom{aaaa}  \end{eqnarray}
here we designate $\widehat{a}_{n+1}\!=\!\widehat{y}_{0}, \
\widehat{b}_{n+1}\!=\!\widehat{y}\,_{T}$ in the record of the "big" 
conjunction for the sake of brevity.

\begin{prop}\label{Pr5} The formula $\Omega(X,P)$ has the property (ii) from 
the statement of Theorem~\ref{main}. In other words, this sentence is true on 
the Boolean algebra $\mathcal{B}$ if and only if the machine $P$ accepts the 
input $X$ within $T$ steps. \end{prop}

\begin{proof} Let $\Theta_s\!=\!\Theta_s(\widehat{a}_{s},\widehat{a}_{s+1},
\widehat{b}_{s},\widehat{b}_{s+1})$ be a denotation for a disjunction of 
equa\-li\-ties \ $(\widehat{a}_{s+1}\!\approx\!\widehat{a}_{s}\wedge
\widehat{v}_{s}\!\approx\!\widehat{b}_{s})\vee(\widehat{v}_{s}\!\approx\!
\widehat{a}_{s}\wedge\widehat{b}_{s}\!\approx\!\widehat{b}_{s+1})$.
If we carry the quantifiers through the subformulae, which do not contain 
the corresponding variables (recall that a conjunction connects more intimately  
than an implication according to the agreement of Subsection~\ref{s3.3}), then we 
will obtain that the part of the formula $\Omega(X,P)$, which is located in the 
big square brackets in (1), is equivalent to each of the three following 
formulae:
\begin{eqnarray*} 1) \ \chi(0)(\widehat{y}_{0}) \ \& \
\exists\,\widehat{v}_{n}\forall\,\widehat{a}_{n}\forall\,\widehat{b}_{n}
\ldots\exists\,\widehat{v}_1\forall\,\widehat{a}_{1}\forall\,\widehat{b}_{1}
\bigl\{\bigwedge\limits_{1\leq\,s\leq\,n}\Theta_{s}\rightarrow
\Phi^{(0)}(P)(\widehat{a}_{1},\widehat{b}_{1})\bigr\}; \hphantom{aaaaa} \\[2mm]
2) \ \Psi K(0)(\widehat{y}_{0}) \ \& \ \exists\,\widehat{v}_{n}\forall\,
\widehat{a}_{n}\forall\,\widehat{b}_{n}\ldots\exists\,\widehat{v}_1\forall\,
\widehat{a}_{1}\forall\,\widehat{b}_{1}\bigl\{\Theta_{n}\rightarrow
[\Theta_{n-1}\to(\ldots \rightarrow  \hphantom{aaaaaaaa}\\%
\rightarrow\{\Theta_{1}\to\Phi^{(0)}(P)
(\widehat{a}_{1},\widehat{b}_{1})\}\ldots)] \bigr\}; \hphantom{aaa}\\[2mm]
3) \ \Psi K(0)(\widehat{y}_{0}) \ \& \ \exists\,\widehat{v}_{n}
\forall\,\widehat{a}_{n}\forall\,\widehat{b}_{n}\bigl\{\Theta_{n}
\rightarrow\exists\,\widehat{v}_{n-1}\forall\,\widehat{a}_{n-1}\forall\,
\widehat{b}_{n-1}[\Theta_{n-1}\to(\ldots \rightarrow  \hphantom{aaaaa}\\
\rightarrow\exists\,\widehat{v}_1\forall\,\widehat{a}_{1}\forall\,
\widehat{b}_{1}\{\Theta_{1}\to
\Phi^{(0)}(P)(\widehat{a}_{1},\widehat{b}_{1})\})]\ldots\bigr\}. \hphantom{aaaa}
\end{eqnarray*}

Under to the definition, the formula \ 
$\exists\,\widehat{v}_s\forall\,\widehat{a}_{s}\forall\,\widehat{b}_{s}
(\Theta_{s}\to\Phi^{(s-1)}(P)(\widehat{a}_{s},\widehat{b}_{s}))$ 
contracts into $\Phi^{(s)}(P)(\widehat{a}_{s+1},\widehat{b}_{s+1})$. 
Therefore the whole $\Omega(X,P)$ is equi\-va\-lent to  \
$\forall\widehat{y}_{0},\widehat{y}\,_{T}\bigl[\bigl(\Psi K(0) \
\& \ \Phi^{(n)}(P)\bigr)\to\chi(\omega)\bigr]$. Consequently, based on 
Proposition \ref{Pr4}(ii) and Lemma \ref{lICon}(i), one could say that formula (1)  
is the mo\-de\-ling formula. \end{proof}

\subsection{The time of writing of $\Omega(X,P)$}\label{s6.4}
The simulating formula $\Omega(X,P)$ is described by the definition (1) in an 
explicit form, this allows us to design an algorithm for its construction. 
It remains only to prove the properties (i) and (iii) of the statement of 
Theo\-rem~\ref{main}. Before we substantiate the polynomiality of the algorithm, 
we will make sure that the formula $\Omega(X,P)$ of a form (1) 
has a polynomial length. We recall that the length of a formula is calculated 
in the natural language --- see Subsections~\ref{s3.1} and \ref{s3.2}.

\begin{lem}\label{lOm1} There exists a constant $D\!>\!0$ such that it does 
not depend on $P$ and $n$ and the inequalities \
$|X|\!=\!n\leqslant|\Omega(X,P)|\leqslant\!D\cdot|P|\cdot|X|^{2+\varepsilon}$ 
\ hold for all the long enough $X$ and any preassigned \ $\varepsilon\!>\!0$. 
\end{lem}

\begin{proof} Many components of the modeling formula were estimated already 
du\-ring their description, but their lengths were estimated on the 
assumption that their subformulae are written with basic variables 
$\langle\widehat{x}_{t},\widehat{q}_{t},\widehat{z}_{t},\widehat{d}_{t},
\widehat{f}_{t}\rangle$, which were denoted in Subsection \ref{s4.3} as 
$\widehat{y}_{t}$. However, they are not included in the composition of the
subformulae of $\Phi^{(0)}(P)(\widehat{a}_{1},\widehat{b}_{1})$  --- we have 
written the variables from the tuples $\widehat{a}_{1}$ and $\widehat{b}_{1}$ 
instead of theirs. Namely, the first $n\!+\!1$ variables in the tuple 
$\widehat{a}_{1}$ serve as $\widehat{x}_{t}$, and they serve as $\widehat{x}_{t+1}$ 
in the tuple $\widehat{b}_{1}$; the second $r\!+\!1$ variables in $\widehat{a}_{1}$ 
are put instead of $\widehat{q}_{t}$, and they are put in place of 
$\widehat{q}_{t+1}$ in $\widehat{b}_{1}$ and so on. Certainly, 
this replacement does not influence on the length of those formulae, 
where the variables are located, if one disregards the length of indices. 

However, the length of the indices has changed markedly. Just because of 
this reason, they were earlier taken into account only implicitly for the 
estimation of the lengths of the formulae, e.g., see Lemmata \ref{ll2} and 
\ref{lphi}, or they were not counted at all (see Lemma \ref{lcl}).

The first indices of variables of the form $\widehat{a}_{1}$ and
$\widehat{b}_{1}$ are $\lfloor\lg 1\rfloor\!+\!1\!=\!1$ in length. The second 
subscripts of these variables have their lengths restricted from above by
$E\rightleftharpoons\lfloor\lg(2n\!+\!3r\!+\!5)\rfloor\!+\!1$. The second 
indices of variables $\widehat{y}_{0}$ are shorter --- they are bounded by
$E_{0}\rightleftharpoons\lfloor\max\{\lg n,\lg(r)\}\rfloor\!+\!1$; besides, 
the subscripts are not included in the record of the tuples of constants. The  
number $n\!=\!|X|$ will grow bigger than $r$, and so the inequality 
$E,E_{0}\leqslant\lfloor\lg n\rfloor\!+\!2$ holds for the long enough $X$. 
Therefore by Lemmata \ref{lcl} and \ref{lphi}, the quasi-equations and 
timers of the subformulae $\chi(0)(\widehat{y}_{0})$ and 
$\Phi^{(0)}(P)(\widehat{a}_{1},\widehat{b}_{1})$ from (1), in which the 
tuples $\widehat{u}(\beta)$ are not included for $\beta\!=\!R,L$, are not 
greater than $D_{3}\cdot n\cdot(\lfloor\lg n\rfloor\!+\!2)$ in length; the 
clauses and timers comprising $\widehat{u}(\beta)$ have a length not more 
than $D_{4}\cdot[n\cdot(\lfloor\lg n\rfloor\!+\!2)]^{2}$ for the suitable 
constants $D_{3}$ and $D_{4}$. By Lemma \ref{lphi} we have \
$|\varphi(k)|\!\leqslant\! D_{5}\cdot[n\cdot(\lfloor\lg n\rfloor\!+\!2)]^{2}$, 
but with another constant $D_{5}$ and for the long enough $X$.

The system of equalities, which are under the "big" conjunction in (1), is \ 
$\mathcal{O}(n\cdot[n\cdot(\lfloor\lg n\rfloor\!+\!2)])$ in length; and the 
quantifier prefix, which is situated before this conjunction, has 
approximately the same length. It is easy to notice that an inequality \ 
$(\lfloor\lg n\rfloor\!+\!2)^{2}\leqslant\!n^{\varepsilon}$ holds for all \ 
$\varepsilon\!>\!0$ and the big enough $n$. It follows from this and Lemma 
\ref{trcl}(ii) that \ $|\Omega(X,P)|\!\leqslant\!D\cdot|P|\cdot|X|^{2+
\varepsilon}$ for some constant $D$.\end{proof}

\begin{cor}\label{lOm2} There is a constant $D_{\varepsilon}$ such that the 
inequality \ $|\Omega(X,P)|\!\leqslant\!D_{\varepsilon}\cdot|P|\cdot|X|^{2+
\varepsilon}$ holds for all $X$ and $P$. \end{cor}

\begin{proof} For each given $\varepsilon$, there exists only a finite number 
of the strings $X$, for which the inequality from the statement of the lemma 
can be violated. Therefore the ratio \
$\lceil|\Omega(X,P)|/(|P|\cdot|X|^{2+\varepsilon})\rceil$ attains 
its maximum for these $X$. Clearly, it is fit for our $D_{\varepsilon}$. 
\end{proof}%

\begin{cor}\label{tOm} There exists a polynomial $g$ such that for all $X$ and 
$P$ the construction time of the sentence $\Omega(X,P)$ is not greater than \ 
$g(|X|+|P|)$. \end{cor}

\begin{proof} We will at first, estimate the time needed for a multi-tape 
Turing machine $P_{1}$ to write down the formula $\Omega(X,P)$. The running
alphabet of this machine contains all symbols of natural language from 
Subsection~\ref{s3.1}.

Let the input tape of the machine comprises a string $X$ and a program $P$, 
and $|A|$ be a quantity of different symbols in the record of $X$ and  
$P$. The machine $P_{1}$ can determine a length $n$ of the input $X$, a 
maximal number $U$ of internal states in $P$, and a size of $|A|$ during 
one passage along its input tape. The calculation of the values of 
$r\!\leqslant\!\log_{2}(U\!+\!1\!+\!|A|)$ and the decimal notation of it and 
$n$ takes a time bounded by a polynomial of $|P|$ and $|X|$. Further, the 
$P_{1}$ moves again along the record of $X$ and $P$ and writes the formula 
$\chi(0)(\widehat{y}_{0})$ at first, after that it writes $\Phi^{(0)}(P)
(\widehat{a}_{1},\widehat{b}_{1}))$, and finally, it designs $\Omega(X,P)$. It 
is clear that this process takes the time, which is no greater than the value 
of $p(|X|+|P|)$ for some polynomial $p(y)$. 

The single-tape variant $P_{2}$ of the machine $P_{1}$ will do the same 
actions in the time equal to $g(|X|+|P|)$, which is of the form of
$\mathcal{O}([p(|X|+|P|)]^{2})$  \cite{AHU,ArBar}. \end{proof}

\section{The complexity of the theory of a single equivalence relation}\label{s7}

Let $\mathfrak{K}$ be a class of the algebraic systems, whose signature (or 
underlying language) $\sigma$ contains the symbol of the binary predicate 
$\backsim$, and this predicate is interpreted as an equivalence relation on 
every structure of the class, in particular, $\backsim$ may be an equality 
relation. We denote these relations by the same symbol. 
 
\begin{defin}\label{eqntr}{\rm Let us assume that there exists a 
\emph{$\backsim$-nontrivial} system $\mathcal{E}$ in a class $\mathfrak{K}$, 
namely, such a structure that contains at least two $\backsim$-nonequivalent  
elements. Then the class $\mathfrak{K}$ is also termed 
\emph{$\backsim$-nontrivial}. A theory $\mathcal{T}$ is named 
\emph{$\backsim$-nontrivial} if it has a $\backsim$-nontrivial model. When 
$\backsim$ is either the equality relation or there is such formula $N(x,y)$ 
of the signature $\sigma$ that the sentence $\exists x,yN(x,y)$ is consistent 
with the theory $Th(\mathfrak{K})$ (or $\mathcal{T}$, or belongs to 
$Th(\mathcal{E})$), and this formula senses that the elements $x$ and $y$ 
are not equal, then we will replace the term "$\backsim$-nontrivial" with 
\emph{"equational-nontrivial"}.} \end{defin}
 
\begin{thm}\label{Teqntr} Let $\mathcal{E}$, $\mathfrak{K}$, and $\mathcal{T}$ 
accordingly be a $\backsim$-nontrivial system, class, and theory of the 
signature $\sigma$, in particular, they may be equational-nontrivial. Then \ 
there is an algorithm such that for every program $P$ and any input string 
$X$, builds the sentence $\Omega^{(T)}(X,P)$ of the signature $\sigma$, 
where $T\!\in\!\{Th(\mathcal{E}),Th(\mathfrak{K}),\mathcal{T}\}$; \  
this formula possesses the properties (i) and (ii) of the word $S(P,X)$ from 
the statement of  Proposition~\ref{comRecL}, where $\mathcal{L}\!=\!T$; 
$F(|X|)\!=\!\exp(|X|)$. Moreover, for each $\varepsilon\!>\!0$, there exists a 
constant $E_{T,\sigma}$ such that the inequality \ $|\Omega^{(T)}(X,P)|
\leqslant\!E_{T,\sigma}\cdot|P|\cdot|X|^{2+\varepsilon}$ \ holds for any 
long enough $X$. \end{thm}  

\begin{proof} At first, for given $X$ and $P$, we write a simulating sentence 
$\Omega(X,P)$ of the theory of the Boolean algebra $\mathcal{B}$ in the 
signature $\langle\cap,\cup,C,0,1\rangle$ with the equality symbol $\approx$, 
applying Theorem~\ref{main}. Then, we will transform it in the required formulae 
$\Omega^{(T)}(X,P)$ within polynomial time.%

For the sake of simplicity of denotations, we assume that the $\backsim$-
nontrivial structure $\mathcal{E}$ is a model for the theory $\mathcal{T}$, 
belongs to the class $\mathfrak{K}$, and has the signature $\sigma$.%

Let $\varphi$ be a sentence of Boolean signature. We will construct the closed 
formulae $\varphi^{(2.j)}$ so that \ $\mathcal{B}\!\models\!\varphi
\Leftrightarrow\mathcal{E}\!\models\!\varphi^{(2.j)}$, where $j$ can be 0,1, 
or 2 depending on the signature $\sigma$. 

\underline{Case 0.} The signature $\sigma$ contains the equivalence symbol 
$\backsim$ and the two constant symbols $c_{0}$ and $c_{1}$ such that 
$\mathcal{E}\!\models\!\neg c_{0}\!\backsim\!c_{1}$. In the first stage, we 
accordingly replace each occurrence of the subformulae of the kind \ $\exists 
y\psi$; $\forall x\psi$; $t\!\approx\!u$ \ with the formulae $\exists y((y\!
\backsim\!c_{0}\vee y\!\backsim\!c_{1})\wedge\psi)$; $\forall x((x\!\backsim\!
c_0\vee x\!\backsim\!c_1)\!\rightarrow\!\psi)$; $t\!\backsim\!u$, where $t$ 
and $u$ are the terms. 

We carry out the second stage's transformations during several passages until 
the formula ceases to change. In this stage, \ a) we replace the subformulae of 
the kind $C(t)\!\backsim\!u$ and $t\!\backsim\!C(u)$ with the formula $\neg\, 
t\!\backsim\!u$. \ If a term $u$ is not the constant 0 or 1, then we replace: \ 
b) the subformulae of the kind \ $t_1\!\cup\!t_2\!\cup\!\ldots\!\cup\!t_s
\backsim\!u$ and $u\backsim\!t\!\cup\!t_1\!\cup\!t_2\!\cup\!\ldots\!\cup\!t_s$ 
\ with the formula \ $[(t_1\!\backsim\!c_{1}\vee t_2\!\backsim\!c_{1}\vee
\ldots\vee t_s\!\backsim\!c_{1})\!\rightarrow\!u\!\backsim\!c_{1}]
\wedge[(t_1\!\backsim\!c_{0}\wedge t_2\!\backsim\!c_{0}\wedge\ldots\wedge t_s
\!\backsim\!c_{0})\rightarrow u\!\backsim\!c_{0}]$; and \ c) the subformulae \ 
$t_1\!\cap\!t_2\!\cap\ldots\!\cap\!t_s\!\backsim\!u$ and $u\!\backsim\!t_1\!
\cap\!t_2\!\cap\ldots\!\cap\!t_s$ \ with the formula \ $[(t_1\!\backsim\!
c_{0}\vee t_2\!\backsim\!c_{0}\vee\ldots\vee t_s\!\backsim\!c_{0})\!
\rightarrow\!u\!\backsim\!c_{0}]\wedge[(t_1\!\backsim\!c_{1}\wedge t_2\!
\backsim\!c_{1}\wedge\ldots\wedge t_s\!\backsim\!c_{1})\rightarrow u\!\backsim 
\!c_{1}]$.We complete the second stage by replacing the constants 0 and 1 with 
the constants $c_0$ and $c_1$, respectively. %

After the second stage, the resulting record can contain symbols $\cup,\ \cap$ 
of the signature of Boolean algebras. In the third stage, we replace 
accordingly each occurrence of the subformulae of the kind \ $t_1\!\cup t_2
\cup\ldots\!\cup t_s\!\backsim\!c_{1}$, \ $t_1\!\cap\!t_2\!\cap\ldots\!\cap
\!t_s\!\backsim\!c_{1}$, \ $t_1\!\cup\!t_2\!\cup\!\ldots\!\cup\!t_s\!\backsim
\!c_{0}$, \ $t_1\!\cap\!t_2\!\cap\ldots\!\cap\!t_s\!\backsim\!c_{0}$ \ with 
the formulae \ $t_1\!\backsim\!c_{1}\vee t_2\!\backsim\!c_{1}\vee\ldots\vee 
t_s\!\backsim\!c_{1}$, \ $t_1\!\backsim\!c_{1}\wedge t_2\!\backsim\!c_{1}
\wedge\ldots\wedge t_s\!\backsim\!c_{1}$, \ $t_1\!\backsim\!c_{0}\wedge t_2\!
\backsim\!c_{0}\wedge\ldots\wedge t_s\!\backsim\!c_{0}$, \ 
$t_1\!\backsim\!c_{0}\vee t_2\!\backsim\!c_{0}\vee\ldots\vee t_s\!\backsim\!
c_{0}$.

We execute these transformations as long as the record contains at least one 
symbol of the signature of Boolean algebras. The number of such symbols is 
decreased at least by one on every passage for the second and third stages, and 
the first stage can be realized on the only passage. So we need at most $n$ 
passages, where $n$ is a length of the sentence $\varphi$. The length of the 
whole record grows linearly on each pass, since the transformation of the kind 
b) or c) of the second stage is longest, but it increases the length no more 
than five times (for $s=2$).  

Nevertheless, the length of the resulting record \ $\varphi^{(2.0)}
\rightleftharpoons\varphi^{\backsim}_{c_0,c_1}$ can increase non-linearly in 
common case. For instance, if $\varphi$ contains an atomic formula of the kind 
$$\bigcup_{i}\Bigl\{\bigcap_{j}\Bigl[\bigcup_{k}\Bigl(\ldots\Bigr)\Big]
\Big\}\approx u,$$ 
where the number of the alternations of the "big" conjunctions and 
disjunctions depends on $n$. 
  
However, there are not such subformulae in the sentence $\Omega(X,P)$  
simulating for the theory of algebra $\mathcal{B}$. Indeed, under denotation of 
Subsections \ref{s3.3} and \ref{s4.2}, the conversion of the subformulae of the kind 
\ $\widehat{x}_{t}\!\approx\!\widehat{u}(\beta)$ (this is the system of 
equalities) and $\neg\,\widehat{w}\!\approx\!\widehat{u}(\beta)$ (this is the 
disjunction of inequalities) make the most increase if $\beta\!\in\!\{R,L\}$,
because they comprise the atomic formulae of the form \ $x_{t,j}\!\approx\!
u_{j}\oplus u_{j+1}^{\beta}\cdot\ldots\cdot u^{\beta}_{n}$ \ and \ $\neg\,w_{j}
\!\approx\!u_{j}\oplus u_{j+1}^{\beta}\cdot\ldots\cdot u^{\beta}_{n}$, where 
$u_{k}^{\beta}$ is either $u_k$ for $\beta\!=\!R$ or $Cu_{k}$ for $\beta\!=\!L
$ . We recall that these subformulae are \ 
$x_{t,j}\!\approx\![u_{j}\cap C(u_{j+1}^{\beta}\cap\ldots\cap u^{\beta}_{n})]
\cup[Cu_{j}\cap u_{j+1}^{\beta}\cap\ldots\cap u^{\beta}_{n}]$ \ and \ $\neg\, 
w_{j}\!\approx\![u_{j}\cap C(u_{j+1}^{\beta}\cap\ldots\cap u^{\beta}_{n})]
\cup [Cu_{j}\cap u_{j+1}^{\beta}\cap\ldots\cap u^{\beta}_{n}]$ by our 
denotation. So, we need to perform only the three transformations of the kind 
b) and c) in order to convert the sentence $\Omega(X,P)$ into 
$\Omega(X,P)^{(2.0)}$. Therefore the estimation \ $|\Omega(X,P)^{(2.0)}|
\leqslant D_0|\Omega(X,P)|$ is valid for appropriate constant $D_0$. Since one 
can execute every passage of any stage within $\mathcal{O}(|\Omega(X,P)|^{2})$ 
steps, the entire transformation takes the polynomial time.

\underline{Case 1.} The signature of the structure $\mathcal{E}$ has not two various 
constant symbols. We replace the constants $c_0$ and $c_1$ in the formula   
$\varphi^{(2.0)}\rightleftharpoons\varphi^{\backsim}_{c_0,c_1}$ with the new 
variables $a$ and $b$, respectively. We obtain the formula $\varphi^{\backsim}
_{a,b} $, and write additionally the prefix after that: \ $\varphi^{(2.1)}
\rightleftharpoons\exists a,b[\neg\,a\!\backsim\!b \ \& \ \varphi^{\backsim}
_{a,b}]$. It is clear that $|\varphi^{(2.1)}|\leqslant2|\varphi^{(2.0)}|$ for 
$|\varphi^{(2.0})|\geqslant11$, and so \ $|\Omega(X,P)^{(2.1)}|\leqslant 
D_1|\Omega(X,P)|$ for appropriate constant $D_1$. 

\underline{Case 2.} The signature $\sigma$ does not contain the equivalence 
symbol, but there exists a formula $N(x,y)$, which asserts that the elements $x$ 
and $y$ is not equal. We replace every occurrence of the atomic subformula 
of the kind $t\!\backsim\!s$ in the $\varphi_{a,b}^{\backsim}$ with the 
formula $\neg N(t,s)$ and add the prefix: \ $\varphi^{(2.2)}\rightleftharpoons
\exists a,b[N(a,b) \ \& \ \varphi_{a,b}^{N}]$. It is obvious that 
$|\varphi^{(2.2)}|\leqslant|N(x,y)|\cdot|\varphi^{2.1}|$, hence \ 
$|\Omega(X,P)^{(2.2)}|\leqslant D_2|\Omega(X,P)|$ for some constant $D_2$. 

One can easily prove by induction on the complexity of the formulae that the 
condition $\mathcal{B}\models\varphi$ is equally matched to one of the
following conditions (depending on the signature $\sigma$): either 
$\mathcal{E}\models\varphi^{(2.0)}$, or \ $\mathcal{E}\models\varphi^{(2.1)}$, 
or $\mathcal{E}\!\models\!\varphi^{(2.2)}$. It is clear that if  
$\mathfrak{K}$ and $\mathcal{T}$ are the equational-nontrivial class and  
theory respectively, then the condition $\mathcal{B}\models\varphi$ is also 
equally matched to the conditions $Th(\mathfrak{K})\,\vdash\forall a,b (N(a,b)
\!\rightarrow\!\varphi^{N}_{a,b})$ and $\mathcal{T}
\vdash\!\forall a,b(N(a,b)\!\rightarrow\!\varphi^{N}_{a,b})$.\end{proof}

\begin{cor}\label{comRecEqNtr} The recognition complexity of each  
$\backsim$-nontrivial decidable theory $\mathcal{T}$, in particular, 
equational-nontrivial, has the non-polynomial lower bound, more precisely 
$$\mathcal{T}\notin DTIME(\exp(D_{T,\sigma}\cdot n^{\delta})), \qquad
\textrm{where} \ \ \delta\!=\!(2+\varepsilon)^{-1}, \ \ D_{T,\sigma}\!=\!
(E_{T,\sigma})^{-\delta}.$$ \end{cor}%

\begin{proof} It immediately follows from the theorem and Corollary~\ref{comRecL}.
\end{proof}

\section{Results and Discussion}\label{s8}

Let us notice that nearly all of the decidable theories mentioned in the 
surveys \cite{ELTT,Rab} are nontrivial regarding equality or equivalence. 
So, if we regard "the polynomial algorithm" as a synonym for "the fast-acting 
algorithm", then the quickly decidable theories are almost completely absent. 
Furthermore, the examples given in the introduction and \cite{ComHen,Fish-Rab, 
FlMaSie},\cite{Mey73}--\cite{St74},\cite{Vor} show that the complexity of the 
recognition procedures can be perfectly enormous for many natural, and 
seemingly, relatively simple theories.

It seems plausible that the estimation obtained in Corollary \ref{comRecB} is 
precise enough. One can substantiate this assertion, if firstly, to find the 
upper bound on the recognition complexity of theory $Th(\mathcal{B})$ by the 
multi-tape Turing machines; secondly, to obtain the lower bound for this 
complexity for the same machines. The author suspects that the inequalities of 
Item (iii) of the main theorem are valid as well for the $k$-tape machines, but 
the constant $D$ must be about $k$ times bigger at that. 

Let us point out that the number of the alternation of quantifiers depends on the 
input length in the modeling formulae $\Omega(X,P)$. Therefore the set of these 
sentences does not belong to any class of the polynomial hierarchy. However, if one 
can build the short formulae $\Omega(X,P)$ belonging to $\Delta_{k}^{\mathbf{P}}$ 
for some $k\!>\!1$, then the class $\mathbf{P}$ will be different from this 
$\Delta_{k}^{P}$, hence $\mathbf{P}$ will be not equal to $\mathbf{NP}$ \cite{M-S}.

\subsection{The totality and locality of the simulating methods}\label{s8.1}
The method of Cook's formulae has arisen for the modeling of the 
nondeterministic Turing machine actions within polynomial time, and the 
construction of Stockmeyer and Meyer is also applicable for the same simulation 
in polynomial space, provided that the running time of the machine is 
exponential. This is a significant advantage of these techniques. 

However, our method of modeling by means of formulae is ineligible for 
nondeterministic machines. More precisely, such modeling formulae must be 
exponential in length, when the machine runs in exponential space within the 
exponential time. Unfortunately, the corresponding example is too cumbersome for 
this paper. This example rests on the simple fact that if we set the values of 
the basic color $t$ variables, then we can "see" only at most two the tape 
squares ($\widehat{x}_t$th and maybe $\widehat{z}_t$th) when we are situated 
within the framework of our approach --- see the proof of Proposition \ref{Pr2}. 
So our simulating method is {\it strictly local, pointwise}. At the same time, 
the techniques of Cook and Stockmeyer and Meyer are {\it total}, since they allow 
us to "see" the values of all of the tape cells simultaneously at any instant, 
if we transform the formulae from \cite{Ck} and \cite{S-M} to the "unfolded" 
tantamount form. 

The author is sure that the technique of the direct encoding of machines 
continues to be a potent tool for investigating the computational complexity 
of theories, despite the emergence of other powerful approaches for obtaining 
the lower bounds on this complexity such as the Compton and Henson method 
\cite{ComHen} or the method  of the bounded  concatenations of  Fleischmann, 
Mahr, and  Siefkes \cite{FlMaSie}.  

Nevertheless, the author agrees with the opinion that the coding of the 
machine computations into the models of the theory being studied is a very 
difficult task in many cases. Such coding is partly like the modeling of 
the machine actions with the aid the defining relations, when one wants to 
prove the insolubility of some algorithmic problem for the finitely presented
algebraical structures of given variety (see, for instance, \cite{BaumsGS,Lat}).

In both cases, we have strong restrictions, which are dictated by the necessity 
to be within the framework of the given signature or variety. 
But the case of the algorithmic problem for the finitely presented structures 
is, perhaps, somewhat easier than the simulation of computations by means of the 
formulae of a certain theory. In the first case, we can apply the suitable 
words consisting of the generators of the algebraical system for the 
description of tape configurations or their parts. The value of these words can 
change depending on the defining relations and the identity of variety. However, 
these changes have the local character relatively of the entire structure; 
whereas the variables can take on any values inside the system when we make a 
simulation in the second case.   

The task becomes slightly easier if there are some constants in the theory 
signature. Just for this reason, we work with the Boolean algebra having two 
elements, but not with the language $TQBF$ consisting of the true quantified 
Boolean formulae.

Note also, that the actions simulation of the computational mechanisms, which 
is realized in \cite{BaumsGS} and \cite{Lat} (these devices are the Minsky 
machines in the former, and they are the Minsky operators algorithms in the 
latter), is total. On the other hand, this modeling is somewhat like the 
Compton and Henson method too. Indeed, in all of these cases, the coding of 
computations is done once and for all. In \cite{ComHen}, this is made for 
Turing machines in proving the inseparability results; then, the authors 
transfer the obtained lower bounds from one theory to another, using 
interpretations. Both in \cite{BaumsGS} and in \cite{Lat}, such simulation is 
made in proving the insolubility of the words problem for the appropriate 
module over a certain integral domain; afterward, this module is 
(accordingly isomorphically and homomorphically) embedded in the 
solvable group under construction.

\subsection{Entirely simultaneous and \\ conventionally consistent modeling}\label{s8.2}
Let us investigate the question of the similarity and difference of the modeling 
formulae in the Stockmeyer and Meyer method and those that are described
in this paper. %

On the one hand, it may regard that the formula $\Omega^{(0)}(X,P)(\widehat{y}
_{t},\widehat{y}_{t+1})$ is the analog of the Cook's method formula $A_{0,m}
(\widetilde{U},\widetilde{V})$, which was applied in the proof of Theorem 4.3 in 
\cite{S-M}, here $\widetilde{U}$ and $\widetilde{V}$ are the sequences $(u_1,
\ldots,u_m)$ and $(v_1,\ldots,v_m)$ of the Boolean variables and $m\!=\!q(|X|)$ is 
the value of suitable polynomial $q$ on the length of input $X$; this $m$ and our 
$P$ are $n$ and $\mathfrak{M}$ in \cite{S-M}. Indeed, it says in \cite{S-M} that    
the formula $A_{0,m}(\widetilde{U},\widetilde{V})$ is satisfiable if and only if 
the configuration encoded by the formula $v_1\ldots v_m$ follows from the 
configuration that corresponds to $u_1\ldots u_m$ in, at most, one step of $P$.

So, the formula $A_{0,m}(\widetilde{U},\widetilde{V})$ can be considered as 
the conjunction of the subformulae $u_1\ldots u_m$, $v_1\ldots v_m$, which 
describe the adjacent configurations, and also of the subformula that describes 
the transformation from the former to the latter. One can regard that this 
transfer formula has the kind \ $B_t\&C_t\&D_t\&F_t\&G_t\&H_t$, where \ 
$B_t,C_t,D_t,F_{t},G_t,$, and $H_{t}$ \ are the subformulae of the formulae \
$B,C,D,F,G$, and $H$ respectively from the proof of Theorem 1 in \cite{Ck} (see 
also the proof of Theorem 10.3 in \cite{AHU}) and are obtained from 
them by means the restriction of the last formulae on the fixed value of the 
parameter $t$ designating the step number. 

However, the author more likes the following standpoint. The correct analog of 
the formula $A_{0,m}(\widetilde{U},\widetilde{V})$ is $\Phi^{(0)}(P)(\widehat{y}
_{t},\widehat{y}_{t+1})$. In other words, the former equals to \ $A'_{0,m}
\rightleftharpoons B_t\&C_t\&$ $\&D_t\&F_t\&G_{t}\&H_{t}$, i.e., this formula 
simply describes the regulations of the transformation of one configuration to 
another, but does not contain the descriptions of these configurations \ 
$u_1\ldots u_m$ and  $v_1\ldots v_m$. These descriptions can only be in the 
implicit form anyway, because the intermediate configurations on the tape are 
unknown for us, we can know only the initial configuration and the fragment of the 
terminal one. 

One can easily prove by induction that if 
$$B_{s,m}(\widetilde{U},\widetilde{V})\rightleftharpoons u_1\ldots u_m \ \& 
\ A'_{s,m} \ \& \ v_1\ldots v_m,$$ then $\exists\widetilde{U}\exists\widetilde{V}
B_{s,m}(\widetilde{U},\widetilde{V})$ is true if and only if the configuration 
encoded by $v_1\ldots v_m$ follows from the configuration that corresponds to 
$u_1\ldots u_m$ in, at most, $\exp(s)$ steps of $P$. So this 
$B_{s,m}(\widetilde{U},\widetilde{V})$ is the simulating formula in \cite{S-M}, 
and it is analog of our $\Omega^{(s)}(X,P)(\widehat{y}_{t},\widehat{y}_{t+e(s)})
$. It is quite clear that if $B_{s,m}(\widetilde{U},\widetilde{V})$ is 
satisfiable, then $A'_{s,m}(\widetilde{U},\widetilde{V})$ is the same. In addition, 
the converse is also true for $s\!=\! 0$. This is easily seen from the description 
of formulae $B,C,D,F,G$, and $H$ given in \cite{Ck} (see also the description of 
their analogs $A,B,C,D,E$ in \cite{AHU}), since these 
formulae contain all components of the tuples $\widetilde{U}$ and $\widetilde{V}$.

In both this cases, the simulating formulae $A_{0,m}$ and $B_{0,m}$ have the form 
 $$configuration(t)\ \&\ configuration(t+1)\ \& \ step(t+1),$$
therefore one can say that the {\it entirely simultaneous modeling} of actions has 
been applied in \cite{Ck,S-M}.

But the formula $\Omega^{(0)}(X,P)(\widehat{y}_{t},\widehat{y}_{t+1})$ is 
constructed in another way. It asserts that if the descriptions of the $t$th 
step's configuration (the formula $\Psi K(t)(\widehat{y}_{t})$) and of the 
step $t+1$ ($\Phi^{(0)}(P)(\widehat{y}_{t},\widehat{y}_{t+1})$) are 
correct, then the configuration, which appeared after this step, will be 
adequately described as well (by $\Psi K(t\!+\!1)(\widehat{y}_{t+1})$). We 
call this approach as {\it conventionally consistent modeling} of actions, 
i.e., $\Omega^{(0)}(X,P)$ has such structure: 
$$configuration(t)\ \&\ step(t+1) \ \to \ configuration(t+1).$$ 
  
Thus, the designs of the formulae $\Omega^{(0)}(X,P)(\widehat{y}_{t},\widehat{y}
_{t+1})$ and $A_{0,m}(\widetilde{U},\widetilde{V})$ are essentially different, 
if even one does not take into consideration the presence of the inner 
quantifiers in the former. Furthermore, their free variables "demand" the 
quantifiers of the various kind forr the formulae to become true.

Let us notice that the conventionally consistent modeling is also used in 
\cite{BaumsGS} and \cite{Lat}. Recall in this connection that the investigation of 
the finitely presented algebraic system, which is given with the aid of the 
generators $g_1,\ldots,g_k$ and the defining relations \ 
$R_1(g_1,\ldots,g_k),\ldots,$ $R_m(g_1,\ldots,g_k)$, is equivalent (in many  
respects) to the study of the formulae of the kind 
$$\forall g_1\ldots\forall g_k[(R_1(g_1,\ldots,g_k)\&\ldots\&R_m(g_1,
\ldots,g_k))\rightarrow S(g_1,\ldots,g_k)].$$  
In addition, we saw in Proposition \ref{Pr2}(i) that $\Omega^{(0)}(X,P)$ can 
simultaneously model too, but existential quantifiers are applied at that.
 
\subsection{Open problems} 
It is well known that the theory of two equivalence relations is not decidable, 
but the theory of one such relation $\sim$ is decidable \cite{ELTT}. 
Now it turns out according to Corollary \ref{comRecEqNtr} that although it is 
decidable, but for a very long time.

What will happen if we add some unary predicates or functions to the signature 
with the only equivalence symbol $\sim$ so that the resulting theory remains 
decidable? Will it be possible to find such functions and/or predicates in order  
that the recognition complexity "smoothly" increases? We can formulate this in a 
more precise way.

\begin{prm} Let $\sigma_0,\sigma_1,\ldots$ be a sequence of signatures such that 
$\sigma_0\!\supseteq\!\{\sim\}$ or $\sigma_0\!\supseteq\!\{\approx\}$ and 
$\sigma_i\subset\sigma_{i+1}$ for each natural $i$. Does there exist a sequence 
of the algebraical structures $\mathfrak{M}_0,\mathfrak{M}_1,\ldots$ such that 
their signature accordingly are $\sigma_0,\sigma_1,\ldots$ and $$Th(\mathfrak{M}
_j)\!\in DTIME(\exp_{j+2}(n)) \ \setminus \ DTIME(\exp_{j+1}(n))?$$ \end{prm}

Recall that $Th(\mathfrak{M})$ denotes the first-order theory of the system 
$\mathfrak{M}$. It is possible that there already is a candidate for the like 
sequence of the higher-order theories with the "smoothly" increasing  
recognition complexity.

\begin{prm} Let $\Omega^{(k)}$ be a fragment of the type theory $\Omega$ from 
\cite{Vor}, this fragment is obtained with the aid of the restriction of the types 
of variables by level $k$. Can one point out, for each natural $k$, such a number 
$s$ that \ $\Omega^{(k)}\!\in DTIME(\exp_{k+s+1}(n)) \ \setminus \ DTIME(\exp_{k+s}
(n))$? \end{prm}

\begin{prm} It seems quite plausible that the theory of finite Boolean algebras 
has a double exponential as the lower bound on the complexity of recognition.
\end{prm}

\begin{prm} Let $F(n)$ be a limit upper bound for all polynomials (see 
Definition \ref{dLUBP}).
What algebraic and/or model-theoretic properties must be possessed 
an algebraical structure $\mathfrak{A}$ so that \ $Th(\mathfrak{A})\!\in\!
DTIME(F(n))$ or  $Th(\mathfrak{A})\!\notin\!DTIME(F(n))$ holds?\end{prm} 

\bibliography{mybibfile}

\end{document}